\documentclass[10pt, reqno]{amsart}
\usepackage{graphicx, amssymb, amsmath, amsthm, color}
\usepackage{hyperref}

\allowdisplaybreaks

\numberwithin{equation}{section}

\newcommand{\uhi}{u_{\text{\upshape hi}}}
\newcommand{\ulo}{u_{\text{\upshape lo}}}

\newcommand{\PHi}{P_\text{\upshape hi}}
\newcommand{\PLo}{P_\text{\upshape lo}}

\newcommand{\Pj}{P_{\leq\lambda_j^{-1}}}

\newcommand{\Pjm}{P_{\lambda_j^{-1}<\cdot\leq cN(t)}}
\newcommand{\Pc}{P_{>cN(t)}}

\let\Re=\undefined\DeclareMathOperator*{\Re}{Re}
\let\Im=\undefined\DeclareMathOperator*{\Im}{Im}

\newcommand{\R}{\mathbb{R}}
\newcommand{\C}{\mathbb{C}}
\newcommand{\wh}[1]{\widehat{#1}}

\newcommand{\M}{\mathcal{M}}

\newcommand{\F}{\mathcal{F}}

\newcommand{\Z}{\mathbb{Z}}

\newcommand{\nsc}{|\nabla|^{\frac32}}

\newtheorem{theorem}{Theorem}[section]

\newtheorem{lemma}[theorem]{Lemma}

\newtheorem{corollary}[theorem]{Corollary}

\newtheorem{proposition}[theorem]{Proposition}

\theoremstyle{definition}
\newtheorem{definition}[theorem]{Definition}
\newtheorem{remark}[theorem]{Remark}

\theoremstyle{remark}

\begin{document}

\title[4$d$ quintic NLS]{The defocusing quintic NLS in four space dimensions}
\author[B. Dodson]{Benjamin Dodson}
\address{Department of Mathematics, Johns Hopkins University, Baltimore, MD}
\email{bdodson4@jhu.edu}

\author[C. Miao]{Changxing Miao}
\address{Institute of Applied Physics and Computational Mathematics,
P. O. Box 8009,\ Beijing,\ China,\ 100088}
\email{miao\_changxing@iapcm.ac.cn}

\author[J. Murphy]{Jason Murphy}
\address{Department of Mathematics,
University of California, Berkeley, USA}
\email{murphy@math.berkeley.edu}

\author[J. Zheng]{Jiqiang Zheng}
\address{Universit\'e Nice Sophia-Antipolis, 06108 Nice Cedex 02, France}
\email{zhengjiqiang@gmail.com, zheng@unice.fr}

\begin{abstract} We consider the defocusing quintic nonlinear Schr\"odinger equation in four space dimensions.  We prove that any solution that remains bounded in the critical Sobolev space must be global and scatter. We employ a space-localized interaction Morawetz inequality, the proof of which requires us to overcome the logarithmic failure in the double Duhamel argument in four dimensions.
\end{abstract}

\maketitle


\section{Introduction}
We consider the defocusing quintic nonlinear Schr\"odinger equation (NLS) in four space dimensions: 
\begin{equation}\label{nls}
\begin{cases}
(i\partial_t+\Delta)u=|u|^4 u \\
u(0)=u_0\in\dot H_x^{\frac32}(\R^4),
\end{cases}
\end{equation}
with $u:\R_t\times\R_x^4\to\C$. The equation \eqref{nls} is called $\dot H_x^{\frac32}$-critical because the rescaling that preserves the class of solutions, namely $u(t,x)\mapsto\lambda^{\frac12}u(\lambda^2 t,\lambda x)$, leaves invariant the $\dot H_x^{\frac32}$-norm of the initial data. 

We prove that any solution to \eqref{nls} that remains bounded in the critical Sobolev space, namely $\dot H_x^{\frac32}(\R^4)$, must be global and scatter.  In \cite{MMZ}, we proved the analogous statement for \eqref{nls} with the nonlinearity $|u|^p u$ for $2<p<4$. In this paper, we treat the endpoint $p=4$, where the techniques in \cite{MMZ} break down. 

We start with some definitions. A function $u:I\times\R^4\to\C$ is a \emph{solution} to \eqref{nls} if it belongs to $C_t\dot{H}_x^{\frac32}(K\times\R^4)\cap L_{t,x}^{12}(K\times\R^4)$ for any compact $K\subset I$ and obeys the Duhamel formula
$$
u(t)=e^{i(t-t_0)\Delta}u(t_0)-i\int_{t_0}^t e^{i(t-s)\Delta}(| u|^4 u)(s)\,ds
$$
for each $t,t_0\in I$. We call $I$ the \emph{lifespan} of $u$. We call $u$ a \emph{maximal-lifespan solution} if it cannot be extended to any strictly larger interval. We call $u$ \emph{global} if $I=\R$. 

We define the \emph{scattering size} of a solution $u:I\times\R^4\to\C$ by 
$$
S_I(u):=\iint_{I\times\R^4}| u(t,x)|^{12}\,dx\,dt.
$$
If there exists $t_0\in I$ such that $S_{[t_0,\sup I)}(u)=\infty$ we say that $u$ \emph{blows up forward in time}. If there exists $t_0\in I$ such that $S_{(\inf I,t_0]}(u)=\infty$ we say that $u$ \emph{blows up backward in time}.

If $u$ is global and obeys $S_\R(u)<\infty$, then standard arguments show that $u$ \emph{scatters}, that is, there exist $u_\pm\in \dot{H}_x^{\frac32}(\R^4)$ such that
	$$\lim_{t\to\pm\infty} \|u(t)-e^{it\Delta}u_{\pm}\|_{\dot{H}_x^{\frac32}(\R^4)}=0.$$ 

Our main result is the following theorem.

\begin{theorem}\label{T}
Suppose $u:I\times\R^4\to\C$ is a maximal-lifespan solution to \eqref{nls} such that $u\in L_t^\infty\dot{H}_x^{\frac32}(I\times\R^4).$ Then $u$ is global and scatters, with 
$$
S_\R(u)\leq C\bigl(\|u\|_{L_t^\infty\dot H_x^{\frac32}(\R\times\R^4)}\bigr)
$$
for some function $C:[0,\infty)\to[0,\infty)$.
\end{theorem}

The motivation for Theorem~\ref{T} comes from the global well-posedness and scattering results for the mass- and energy-critical NLS. In space dimension $d$, the equation $(i\partial_t+\Delta)u=\pm|u|^p u$ is \emph{mass-critical} for $p=\frac4d$ and \emph{energy-critical} for $p=\frac{4}{d-2}$.  These cases are distinguished by the presence of a conservation law at the critical regularity.  For the mass-critical case, the critical space is $L_x^2(\R^d)$, and the conserved quantity is the \emph{mass}:
$$
M[u(t)]:=\int_{\R^d} |u(t,x)|^2 \,dx.
$$
For the energy-critical case, the critical space is $\dot H_x^1(\R^d)$, and the conserved quantity is the \emph{energy}:
$$
E[u(t)]:=\int_{\R^d} \tfrac12 |\nabla u(t,x)|^2\pm\tfrac{1}{p+2}|u(t,x)|^{p+2}\,dx.
$$

For the defocusing mass- and energy-critical NLS, arbitrary initial data in the critical space lead to global solutions that scatter; in the focusing case, one has global well-posedness and scattering results below the `ground state' \cite{Bo99a, CKSTT07, Dodson3, Dodson2, Dodson1, Dodson, Dodson14, Grillakis, KM, KTV2009, KV20101, KV3D, KVZ2008, RV, TaoRadial, TVZ2007, Visanphd, Visan2007, Visan2011}.  A major obstacle to solving these problems was the lack of any monotonicity formulae (i.e. Morawetz estimates) that scale like the mass or energy. The key breakthrough was the \emph{induction on energy} method of Bourgain \cite{Bo99a}: by finding solutions that concentrate on a characteristic length scale (and hence break the scaling symmetry of the equation), the available Morawetz estimates can be brought back into the picture, despite their non-critical scaling. These ideas and techniques have been developed extensively in the setting of \emph{concentration compactness} and \emph{minimal counterexamples}, as in the pioneering work of Kenig and Merle \cite{KM}.

A key ingredient for the mass- and energy-critical problems is the \emph{a priori} uniform control over solutions in the critical Sobolev space afforded by conservation laws. For the case of NLS at `non-conserved critical regularity', one has no such \emph{a priori} control; however, the success of the techniques developed to treat the mass- and energy-critical problems suggests that this should be the \emph{only} missing ingredient for a proof of global well-posedness and scattering. Indeed, previous works have shown that critical $\dot H_x^s$-bounds imply scattering for NLS for a range of dimensions and nonlinearities \cite{KM2010, KV2010, MMZ, Mu, Mu2, Mu3, XF}. In \cite{MMZ}, the authors treated the nonlinearity $|u|^p u$ for $2<p<4$ in four space dimensions. In this paper, we address the endpoint $p=4$. 

\subsection{Outline of the proof of Theorem~\ref{T}}
We argue by contradiction, supposing that Theorem~\ref{T} fails. As standard local well-posedness results (via Strichartz estimates and contraction mapping, cf. \cite{Cav, CaW, KVnote}) give global existence and scattering for sufficiently small initial data, we deduce the existence of a critical threshold, below which Theorem~\ref{T} holds but above which we can find solutions with arbitrarily large scattering size. By a limiting argument (see below), we deduce the existence of minimal counterexamples, that is, blowup solutions living at the threshold. As a consequence of minimality, these counterexamples have good compactness properties, specifically, almost periodicity modulo the symmetries of the equation.

\begin{definition}[Almost periodic] A solution $u:I\times\R^4\to\C$ to \eqref{nls} is \emph{almost periodic (modulo symmetries)} if $u\in L_t^\infty\dot H_x^{\frac32}(I\times\R^4)$ and there exist functions $N:I\to\R^+$, $x:I\to\R^4$, and $C:\R^+\to\R^+$ such that for $t\in I$ and $\eta>0$,
$$
\int_{| x-x(t)|>\frac{C(\eta)}{N(t)}}|\nsc u(t,x)|^2\,dx+\int_{| \xi|>C(\eta)N(t)}|\xi|^3|\wh{u}(t,\xi)|^2\,dx<\eta.
$$ 
We call $N(\cdot)$ the \emph{frequency scale}, $x(\cdot)$ the \emph{spatial center}, and $C(\cdot)$ the \emph{compactness modulus}. 
\end{definition}

\begin{remark} Equivalently, $u:I\times\R^4\to\C$ is almost periodic if and only if
$$
\{u(t):t\in I\}\subset\{\lambda^{\frac12}f(\lambda(x+x_0)):\lambda\in(0,\infty),\ x_0\in\R^4,\ f\in K\}
$$
for some compact $K\subset\dot H_x^{\frac32}(\R^4).$ From this, we can deduce that there also exists $c:\R^+\to\R^+$ such that
\begin{equation}\label{APsmall}
\|\nsc u_{\leq c(\eta)N(t)}\|_{L_t^\infty L_x^2(I\times\R^4)}<\eta.
\end{equation}
\end{remark}

The first major step in the proof of Theorem~\ref{T} is the following.

\begin{theorem}[Reduction to almost periodic solutions]\label{T:AP1} If Theorem~\ref{T} fails, then there exists a maximal-lifespan solution $u:I\times\R^4\to\C$ to \eqref{nls} that is almost periodic and blows up in both time directions.
\end{theorem}

As mentioned above, almost periodic solutions are constructed via a limiting argument as minimal blowup solutions. The argument, which has its origin in work of Keraani \cite{Ker0, Ker}, is now considered fairly standard in the field of dispersive equations at critical regularity \cite{KM, KM2010, KTV2009, KV2010, KV20101, KVnote, KVZ2008, Mu, Mu2, TVZ2007}. The argument relies on three main ingredients: a linear profile decomposition for $e^{it\Delta}$ \cite{BeVa, CarKer, Ker0, MerVeg, Shao}, a stability theory for the nonlinear equation (similar to the local theory), and a decoupling statement for nonlinear profiles. Roughly speaking, decoupling means that one can solve the equation (approximately) by decomposing the initial data into profiles, evolving each profile by the nonlinear equation, and then recombining the nonlinear profiles. This step relies essentially on an orthogonality property satisfied by the profiles. In the presence of a non-integer number of derivatives and/or non-algebraic nonlinearities, the decoupling step necessitates some additional technical arguments.  By now, technology exists to treat a  range of these situations \cite{HolRou, KM2010, KV2010, Mu, Mu2}. The necessary arguments for our setting may be found in \cite{KV2010}, which treats general energy-supercritical NLS. For a good introduction to concentration compactness techniques in the dispersive setting, we refer the reader to \cite{KVnote, Visbook}.

We next discuss some further properties of almost periodic solutions. First, the frequency scale obeys the following local constancy property (cf. \cite[Lemma~5.18]{KVnote}).

\begin{lemma}[Local constancy]\label{lem:lc} If $u:I\times\R^4\to\C$ is a maximal-lifespan almost periodic solution, then there exists $\delta=\delta(u)>0$ such that for $t_0\in I$ we have
$$
[t_0-\delta N(t_0)^{-2},t_0+\delta N(t_0)^{-2}]\subset I,
$$
with $N(t)\sim_u N(t_0)$ for $| t-t_0|\leq\delta N(t_0)^{-2}$. 

In particular, modifying the compactness modulus by a multiplicative factor, we may divide the lifespan $I$ into characteristic subintervals $J_k$ on which we can take $N(t)\equiv N_k$ for some $N_k$, with $| J_k|\sim_u N_k^{-2}$. 
\end{lemma}

Lemma~\ref{lem:lc} provides information about the behavior of the frequency scale at blowup (cf. \cite[Corollary~5.19]{KVnote}):

\begin{corollary}[$N(t)$ at blowup]\label{cor:blowup} Let $u:I\times\R^4\to\C$ be a maximal-lifespan almost periodic solution to \eqref{nls}. If $T$ is a finite endpoint of $I$ then $N(t)\gtrsim_u | T-t|^{-\frac{1}{2}}.$ 
\end{corollary}

We can also relate the frequency scale to the Strichartz norms of an almost periodic solution.  
\begin{lemma}\label{L:SB} Let $u:I\times\R^4\to\C$ be a nonzero almost periodic solution to \eqref{nls}. Then 
$$
\int_I N(t)^2\,dt\lesssim \|\nsc u\|_{L_t^2 L_x^4}^2\lesssim \sum_M \|\nsc u_M\|_{L_t^2 L_x^4}^2 \lesssim_u 1+\int_I N(t)^2\,dt
$$
\end{lemma}

To prove Lemma~\ref{L:SB}, we may adapt the proof of \cite[Lemma~5.21]{KVnote}, making use of the Strichartz estimate below (Proposition~\ref{P:Strichartz}). Briefly, $\int_I N(t)^2\,dt$ counts the number of characteristic subintervals in $I$, while the Strichartz norm is $\sim_u 1$ on each such subinterval. 

We now refine the class of almost periodic solutions that we consider. By rescaling arguments as in \cite{KTV2009, KV20101, TVZ2007}, we can guarantee that the almost periodic solutions we consider do not escape to arbitrarily low frequencies on at least half of their maximal lifespan, say $[0,T_{max})$. Using Lemma~\ref{lem:lc} to divide $[0,T_{max})$ into characteristic subintervals $J_k$, we arrive at the following theorem. 

\begin{theorem}[Further reductions]\label{T:AP} If Theorem~\ref{T} fails, then there exists an almost periodic solution $u:[0,T_{max})\times\R^4\to\C$ to \eqref{nls} that blows up forward in time and satisfies
\begin{equation}\label{E:bdd}
u\in L_t^\infty \dot H_x^{\frac32}([0,T_{max})\times\R^4).
\end{equation}
Furthermore, we may write $[0,T_{max})=\cup_k J_k$, where
\begin{equation}\label{E:Ng1}
N(t)\equiv N_k\geq 1\quad\text{for}\quad t\in J_k,\quad\text{with}\quad |J_k|\sim_u N_k^{-2}.
\end{equation}
We classify $u$ according to the following two scenarios:  
\begin{align}\label{E:KF}
& \int_0^{T_{max}} N(t)^{-3}\,dt<\infty \quad (\text{rapid frequency-cascade solution}), \\ 
\label{E:KI}
&\int_0^{T_{max}} N(t)^{-3}\,dt=\infty \quad (\text{quasi-soliton solution}).
\end{align}
\end{theorem}

To complete the proof of Theorem~\ref{T}, it therefore suffices to rule out the existence of almost periodic solutions as in Theorem~\ref{T:AP}.

The quantity appearing in \eqref{E:KF} and \eqref{E:KI} is related to the interaction Morawetz inequality, an \emph{a priori} estimate for solutions to defocusing NLS introduced in \cite{ckstt}. We recall the estimate here in the four-dimensional setting: defining
$$
M(t)=\iint |u(t,y)|^2 \nabla a(x-y)\cdot\,2\Im(\nabla u(t,x)\bar{u}(t,x))\,dx\,dy,\quad a(x)=|x|,
$$ 
one can prove a lower bound for $\frac{d}{dt}M(t)$ and use the fundamental theorem of calculus to deduce the following (see \cite{RV}, for example):
\begin{equation}\label{IM4}
\int_I\iint_{\R^4\times\R^4}\frac{|u(t,x)|^2|u(t,y)|^2}{|x-y|^3}\,dx\,dy\,dt\lesssim \|u\|_{L_t^\infty L_x^2(I\times\R^4)}^3\|\nabla u\|_{L_t^\infty L_x^2(I\times\R^4)}.
\end{equation}

A scaling argument suggests that for almost periodic solutions to \eqref{nls}, one has
$$
\text{LHS}\eqref{IM4}\sim_u \int_I N(t)^{-3}\,dt,
$$
which explains the appearance of this quantity in Theorem~\ref{T:AP}. In particular, we expect the interaction Morawetz inequality to preclude the possibility of almost periodic solutions satisfying \eqref{E:KI}. We make this heuristic precise in Section~\ref{S:QS} by proving a space-localized interaction Morawetz inequality, which we then use to preclude quasi-solitons (see Proposition~\ref{P:IM}). We need space localization because \eqref{IM4} is not directly applicable in our setting; indeed, we do not control the $H_x^1$-norm of the solutions we consider. 

The proof of the space-localized Morwetz estimate represents the main difficulty of this paper compared to our previous work \cite{MMZ}. Spatial truncation in the standard Morawetz weight introduces error terms that must be controlled to arrive at a useful estimate. To achieve this, we first prove a long-time Strichartz estimate, Proposition~\ref{prop:lts}. Such estimates first appeared in the work of Dodson \cite{Dodson3}, and have since appeared in \cite{Dodson14, KV3D, MMZ, Mu, Mu3, Visan2011}. As in \cite{Dodson14, KV3D}, we use a maximal Strichartz estimate (Proposition~\ref{P:maximal}) to prove the long-time Strichartz estimate, which results in a stronger estimate than the one appearing in our previous work \cite{MMZ}. 

Another key ingredient in the proof of the Morawetz estimate is Proposition~\ref{L:mass}, which is a Strichartz-type estimate used to control the mass of solutions over balls. Proposition~\ref{L:mass} is similar to \cite[Proposition~2.7]{MMZ}, which we used in our previous work for exactly the same purpose. Similar estimates also appear in \cite{Dodson14, KV3D}. To prove Proposition~\ref{L:mass}, we use the double Duhamel argument, which has its origin in \cite{CKSTT07}. In contrast to our previous work, however, we need to put a portion of the nonlinearity in the endpoint $L_t^2 L_x^1$ (cf. the proof of \cite[(6.17)]{MMZ}). Because we are in four space dimensions, this leads to a logarithmic failure in the double Duhamel argument, as was exhibited already in \cite{Dodson14}. 

As in \cite{Dodson14}, we are able to witness the logarithmic loss by rescaling the Morawetz weight by a function of time. This leads to a new error term in the Morawetz estimate, which we handle by adapting arguments of Dodson \cite{Dodson, Dodson14}. In particular, we employ a `smoothing algorithm' to produce a suitable rescaling function, which is closely related to the frequency scale function of the solution. 

Finally, we need to overcome the logarithmic loss stemming from Proposition~\ref{L:mass}. We achieve this by using an appropriate Morawetz weight and exploiting the fact that \eqref{nls} is energy-supercritical, as we now explain. To simply the exposition, let us ignore the time-dependent rescaling (or, equivalently, consider a `true soliton' with $N(t)\equiv 1$). Inspired by \cite{KV3D}, we introduce a weight $a$ such that $a=|x|$ for $|x|\leq R$ and $a$ is constant for $|x|>Re^J$, where $R\gg 1$ and $J\sim\log R$. In the intermediate region, $a$ satisfies $|\partial_r^k a_r|\lesssim_k  J^{-1} r^{-k}$. Some error terms do not require the use of Proposition~\ref{L:mass}, in which case the factor $J^{-1}$ provides smallness. For terms requiring Proposition~\ref{L:mass}, the factor $J^{-1}$ cancels the logarithmic loss; however, we still need to exhibit smallness. For very low frequencies, we can use almost periodicity, as in \eqref{APsmall}. For the remaining higher frequency terms, we use the fact that the error terms involve at most one derivative of the solution, while the equation \eqref{nls} is energy-supercritical; in particular, we control the $\dot H_x^{\frac32}$-norm of the solution, and hence we can exhibit a gain by using Bernstein's inequality.

In our previous work \cite{MMZ}, we proved a space- and frequency-localized interaction Morawetz inequality; in particular, we proved an estimate for the high frequencies of the solution only. This introduced even more error terms, as the high frequencies alone do not solve \eqref{nls}. In this paper, we opt to work with a true solution, which reduces the total number of error terms. Of course, we still need to control the low frequencies in the error terms, but we expect the low frequencies to be relatively harmless due to \eqref{E:Ng1}. To control the low frequencies, we use the long-time Strichartz estimates along with the following proposition, which gives additional decay in the $L_x^q$-sense. 

\begin{proposition}[Additional decay, \cite{MMZ}]\label{P:decay} Suppose $u:[0,T_{max})\times\R^4\to\C$ is an almost periodic solution as in Theorem~\ref{T:AP}. Then
\begin{equation}\label{E:decay}
u\in L_t^\infty L_x^q([0,T_{max})\times\R^4)\quad\text{for all}\quad \tfrac{40}{11}<q\leq 8.
\end{equation}
\end{proposition}
To prove Proposition~\ref{P:decay}, one can argue exactly as in \cite[Proposition~3.1]{MMZ} (in fact, the algebraic nonlinearity $|u|^4 u$ allows for some simplifications). We briefly sketch the ideas here. Defining
$$
f_q(N):=N^{\frac4q-\frac12}\|u_N\|_{L_t^\infty L_x^{q}([0,T_{max})\times\R^4)}\quad (4<q<8),
$$
one uses the reduced Duhamel formula, Strichartz, the dispersive estimate, and a suitable decomposition of the nonlinearity to prove a recurrence relation for $f_q(N)$. Here $N$ is chosen small enough to guarantee that the $L_t^\infty \dot{H}_x^{s_c}$-norm of $u_{\leq N}$ is small; this is possible because of \eqref{E:Ng1}.  As $f_q$ is bounded (by Bernstein and \eqref{E:bdd}), one can combine the recurrence relation with an `acausal Gronwall inequality' to deduce the bound 
$$
\|u_N\|_{L_t^\infty L_x^q}\lesssim_u N^{(2-\frac8q)-}
$$ 
for $4<q<8$ and $N$ small. Interpolation with $\|u_N\|_{L_t^\infty L_x^2}\lesssim_u N^{-\frac32}$ yields
$$
\|u_N\|_{L_t^\infty L_x^{\frac{40}{11}+}}\lesssim N^{0+}$$ 
for $N$ small, which implies the result. See \cite[Section~3]{MMZ} for more details. This type of argument appears originally in \cite{KV20101}, where it was combined with the double Duhamel argument (in dimensions $d\geq 5$) to prove negative regularity.

The remaining scenario in Theorem~\ref{T:AP}, namely, that of rapid frequency-cascades, is comparatively simple. In particular, we use the long-time Strichartz estimate (Proposition~\ref{prop:lts}) together with the following reduced Duhamel formula to show that such solutions are inconsistent with the conservation of mass.
\begin{proposition}[Reduced Duhamel formula]\label{P:RD} Let $u:[0,T_{max})\times\R^4\to\C$ be an almost periodic solution to \eqref{nls} as in  Theorem~\ref{T:AP}. Then for any $t\in [0,T_{max})$,  
\begin{align*}
u(t)=\lim_{T\nearrow T_{max}}i\int_t^T e^{i(t-s)\Delta}\bigl(|u|^4 u\bigr)(s)\,ds
\end{align*} 
as a weak limit in $\dot H_x^{\frac32}(\R^4)$. 
\end{proposition}
The reduced Duhamel formula is a robust consequence of almost periodicity; to prove it, one can adapt the proof of \cite[Proposition~5.23]{KVnote}.  

The rest of this paper is organized as follows. In Section~\ref{S:notation}, we set up notation and collect some useful lemmas, including the Strichartz estimates mentioned above. In Section~\ref{S:lts}, we prove the long-time Strichartz estimate. In Section~\ref{S:rfc}, we preclude the possibility of rapid frequency-cascades. In Section~\ref{S:QS}, we prove the interaction Morawetz inequality and use it to preclude the possibility of quasi-solitons.


\subsection*{Acknowledgements} B.~D. was supported by NSF grant DMS-1500424. C.~M. was supported by the NSFC under grant No. 11171033 and 11231006. J.~M. was supported by the NSF Postdoctoral Fellowship DMS-1400706. J.~Z. was partly supported by the European Research Council, ERC-2012-ADG, project number 320845: Semi-Classical Analysis of Partial Differential Equations. 


\section{Notation and useful lemmas}\label{S:notation}
For nonnegative $X$, $Y$, we write $X\lesssim Y$ to denote $X\leq CY$ for some $C>0$. If $X\lesssim Y\lesssim X$, we write $X\sim Y$. Dependence on certain parameters will be indicated by subscripts; for example, $X\lesssim_u Y$ means $X\leq CY$ for some $C=C(u)$. Dependence of the estimates on the ambient dimension will not be explicitly indicated. We write $\text{\O}(X)$ to denote a finite linear combination of terms that resemble $X$ up to Littlewood--Paley projections, complex conjugation, and/or maximal functions. 

We write $L_t^q L_x^r(I\times\R^4)$ for the Banach space of functions $u:I\times\R^4\to\C$ equipped with the norm
$$
\|u\|_{L_t^q L_x^r(I\times\R^4)}:=\biggl(\int_I \|u(t)\|_{L_x^r(\R^4)}^q\,dt\biggr)^{\frac1q},
$$
with the usual adjustments if $q$ or $r$ is infinite. When $q=r$, we write $L_t^q L_x^q=L_{t,x}^q$. We write $\|f\|_{L_x^r}$ to denote $\|f\|_{L_x^r(\R^4)}$. We write $r'$ to denote the dual exponent to $r$, i.e. the solution to $\frac1r+\frac1{r'}=1.$ 

We define the Fourier transform on $\R^4$ by
$$
\wh{f}(\xi)=\tfrac{1}{(2\pi)^2}\int_{\R^4} e^{-ix\cdot\xi}f(x)\,dx.
$$
For $s\in\R$, we define $|\nabla|^s$ to be the Fourier multiplier operator with symbol $|\xi|^s$, and we define the homogeneous Sobolev norm $\dot H_x^s$ via $ \|f\|_{\dot H_x^s}=\||\nabla|^s f\|_{L_x^2}.$ 

\subsection{Basic harmonic analysis}
Let $\varphi$ be a radial bump function supported on the ball $\{|\xi|\leq \frac{11}{10}\}$ and equal to $1$ on the ball $\{|\xi|\leq 1\}$. For $N\in 2^{\mathbb{Z}}$, we define the Littlewood--Paley projection operators by
\begin{align*}
&\widehat{P_{\leq N}f}(\xi) := \wh{f_{\leq N}}(\xi):=\varphi(\tfrac{\xi}{N})\widehat{f}(\xi), \\ 
&\widehat{P_{> N}f}(\xi) := \wh{f_{>N}}(\xi) := \bigl(1-\varphi(\tfrac{\xi}{N})\bigr)\widehat{f}(\xi), \\ 
&\widehat{P_{N}f}(\xi) := \wh{f_N}(\xi):= \bigl(\varphi(\tfrac{\xi}{N})-\varphi(\tfrac{2\xi}{N})\bigr)\widehat{f}(\xi).
\end{align*}
We also define 
$$
P_{N_1<\cdot\leq N_2}=\sum_{N_1<N\leq N_2} P_N,
$$
where here and throughout such sums are taken over $N\in 2^{\mathbb{Z}}$. 

The Littlewood--Paley operators commute with all other Fourier multiplier operators (such as derivatives and the free propagator), as well as the conjugation operation. These operators are self-adjoint and bounded on every $L_x^r$ and $\dot H_x^s$ space for $1\leq r\leq\infty$ and $s\geq 0$. They also obey the following standard Bernstein estimates.
\begin{lemma}[Bernstein estimates] For $1\leq r\leq q\leq\infty$ and $s\geq 0$,
\begin{align*}
\big\||\nabla|^s P_{\leq N} f \big\|_{L_x^r(\R^4)} & \lesssim N^{s} \big\|P_{\leq N} f \big\|_{L_x^r(\R^4)},  \\
\big\| P_{> N} f \big\|_{L_x^r(\R^4)} & \lesssim  N^{-s} \big\| |\nabla|^{s}P_{> N} f \big\|_{L_x^r(\R^4)}, \\
\big\| P_{\leq N} f \big\|_{L^q(\R^4)} & \lesssim N^{\frac{4}{r}-\frac{4}{q}} \big\| P_{\leq N} f \big\|_{L_x^r(\R^4)}.
\end{align*}
\end{lemma}

We will also need the following fractional chain and product rules from \cite{CW}. 

\begin{lemma}[Fractional calculus, \cite{CW}]\text{ }
\begin{itemize}
\item[(i)] Let $s\geq 0$ and $1<r,r_j,q_j<\infty$ satisfy
$\frac1r=\frac1{r_j}+\frac1{q_j}$ for $j=1,2$. Then
\begin{align*}
\big\||\nabla|^s(fg)\big\|_{L_x^r}&\lesssim\|f\|_{L_x^{r_1}}\big\||\nabla|^s g
	\big\|_{L_x^{q_1}}+\big\||\nabla|^sf\big\|_{L_x^{r_2}}\|g\|_{L_x^{q_2}}.
\end{align*}
\item[(ii)] Let $G\in C^1(\mathbb{C})$ and $s\in(0,1],$ and let $1<r_1\leq \infty$ and $1<r,r_2<\infty$
satisfy $\frac1r=\frac1{r_1}+\frac1{r_2}.$ Then
\begin{equation}\nonumber
\big\||\nabla|^s G(u)\big\|_{L_x^r}\lesssim\|G'(u)\|_{L_x^{r_1}}\big\||\nabla|^su\big\|_{L_x^{r_2}}.
\end{equation}
\end{itemize}
\end{lemma}

\subsection{Strichartz estimates}

The free Schr\"odinger propagator $e^{it\Delta}=\F^{-1}e^{-it|\xi|^2}\F$ is given in physical space by
$$
[e^{it\Delta}f](x)=\tfrac{-1}{16\pi^2 t^2}\int_{\R^4} e^{\frac{i|x-y|^2}{4t}}f(y)\,dy.
$$
It follows that $\|e^{it\Delta}f\|_{L_x^2}\equiv \|f\|_{L_x^2}$ and the following \emph{dispersive estimate} holds:
$$
\|e^{it\Delta}f\|_{L_x^\infty(\R^4)}\lesssim |t|^{-2}\|f\|_{L_x^1(\R^4)}\quad\text{for}\quad t\neq 0.
$$
Interpolation yields
$$
\|e^{it\Delta}f\|_{L_x^q(\R^4)}\lesssim |t|^{-2+\frac{4}{q}}\|f\|_{L_x^{q'}(\R^4)}\quad\text{for}\quad 2\leq q\leq\infty\quad\text{and}\quad t\neq 0.
$$

These bounds imply the standard Strichartz estimates for $e^{it\Delta}$ \cite{GV, KeT98, St}. Arguing as in \cite{CKSTT07}, one can also deduce a `Besov' version. In particular, we have the following estimates.

\begin{proposition}[Strichartz, \cite{CKSTT07, GV, KeT98, St}]\label{P:Strichartz} Let $u:I\times\R^4\to\C$ be a solution to $(i\partial_t+\Delta)u=F$. Then for any $t_0\in I$ and any $2\leq q,\tilde{q},r,\tilde{r}\leq\infty$ satisfying $\frac2q+\frac4r=\frac{2}{\tilde{q}}+\frac{4}{\tilde{r}}=2$, we have
$$
\|u\|_{L_t^q L_x^r(I\times\R^4)}\lesssim \biggl(\sum_N \|u_N\|_{L_t^q L_x^r(I\times\R^4)}^2\biggr)^{\frac12}\lesssim \|u(t_0)\|_{L_x^2(\R^4)} +\|F\|_{L_t^{\tilde{q}'}L_x^{\tilde{r}'}(I\times\R^4)}.
$$
\end{proposition}

As in \cite{CKSTT07, KV3D}, we use this Besov-type Strichartz estimate in order to access the $L_x^\infty$ endpoint. 

\begin{lemma}[Endpoint estimate]\label{lem:endpt}
For $u:I\times\R^4\to\C$,
$$
\|u\|_{L_t^4 L_x^\infty(I\times\R^4)}\lesssim\|\nsc u\|_{L_t^\infty L_x^2(I\times\R^4)}^{\frac{1}{2}}\bigg(\sum_{N}\|\nsc u_N\|_{L_t^2L_x^4(I\times\R^4)}^2\bigg)^{\frac{1}{4}}.
$$  
\end{lemma}

\begin{proof} Using Bernstein and Cauchy--Schwarz, we estimate
\begin{align*}
\|u\|_{L_t^4 L_x^\infty}^4&\lesssim \sum\limits_{N_1\leq \cdots\leq  N_4} \|u_{N_1}\|_{L_{t,x}^\infty} \|u_{N_2}\|_{L_{t,x}^\infty}
	\|u_{N_3}\|_{L_t^2 L_x^\infty} \|u_{N_4}\|_{L_t^2 L_x^\infty} \\
&\lesssim \|\nsc u\|_{L_t^\infty L_x^2}^2 \sum\limits_{N_1\leq\cdots\leq N_4} \bigl(\tfrac{N_1 N_2}{N_3 N_4}\bigr)^{\frac12}
	\|\nsc u_{N_3}\|_{L_t^2 L_x^4} \|\nsc u_{N_4}\|_{L_t^2 L_x^4} \\
&\lesssim \|\nsc u\|_{L_t^\infty L_x^2}^2 \sum_N \|\nsc u_{N}\|_{L_t^2 L_x^4}^2,
\end{align*}
where all space-time norms are over $I\times\R^4$. The result follows.
\end{proof}

We next record a `maximal' Strichartz estimate as in \cite{Dodson14, KV3D}. 

\begin{proposition}[Maximal Strichartz estimate]\label{P:maximal} Let $u:I\times\R^4\to\C$ be a solution to $(i\partial_t+\Delta)u=F+G$. Then for any $t_0\in I$ and $4<q\leq\infty$, we have 
$$
\|\sup_{N} N^{\frac{4}{q}-2}\|P_N u(t)\|_{L_x^q}\|_{L_t^2}\lesssim \||\nabla|^{-1}u(t_0)\|_{L_x^2}+\||\nabla|^{-1}F\|_{L_t^2L_x^{\frac{4}{3}}}+\| G\|_{L_t^2 L_x^1},
$$ 
where all space-time norms are over $I\times\R^4$. 
\end{proposition}

\begin{proof} Beginning with the Duhamel formula
$$
u(t)=e^{i(t-t_0)\Delta}u(t_0)-i\int_{t_0}^t e^{i(t-s)\Delta}F(s)\,ds-i\int_{t_0}^t e^{i(t-s)\Delta}G(s)\,ds,
$$
we estimate the first two terms via Bernstein followed by the $L_t^2 L_x^4$ Strichartz estimate. For the last term, we argue as in \cite{Dodson14}. For completeness, we include the details here. 

We estimate the short-time piece via Bernstein and Strichartz. Letting $\M$ denote the Hardy--Littlewood maximal function, 
\begin{align*}
N^{\frac4q-2}\int_{|t-s|\leq N^{-2}}\|P_N e^{i(t-s)\Delta}G(s)\|_{L_x^q}\,ds &\lesssim N^{2}\int_{|t-s|\leq N^{-2}}
	\|G(s)\|_{L_x^1}\,ds \\
&\lesssim \M(\|G(\cdot)\|_{L_x^1})(t),
\end{align*}
uniformly in $N$. For the long-time piece, we use the dispersive estimate, Bernstein, and the fact that $q>4$ to estimate
\begin{align*}
N^{\frac4q-2}&\int_{|t-s|>N^{-2}}\|P_N  e^{i(t-s)\Delta}G(s)\|_{L_x^{q}}\,ds \\
&\lesssim N^{\frac8q-2}\int_{|t-s|>N^{-2}}|t-s|^{-2+\frac4q}\|G(s)\|_{L_x^1}\,ds \\
&\lesssim \sum_{M>N^{-2}} N^{\frac8q-2} M^{-2+\frac4q}\int_{|t-s|\sim M}\|G(s)\|_{L_x^1}\,ds \\
&\lesssim \sum_{M>N^{-2}} N^{\frac8q-2} M^{-1+\frac4q}\M(\|G(\cdot)\|_{L_x^1})(t) 
\lesssim \M(\|G(\cdot)\|_{L_x^1})(t),
\end{align*}
uniformly in $N$. The result now follows from the maximal function estimate.\end{proof}

Finally, we record the following Strichartz-type estimate, which will play an important role in Section~\ref{S:QS}. This estimate is similar to \cite[Proposition~2.7]{MMZ}. In order to access the $L_t^2 L_x^1$ endpoint, however, we must accept a logarithmic loss. 

\begin{proposition}[Strichartz-type estimate]\label{L:mass} Let $u:I\times\R^4\to\C$ be a solution to $(i\partial_t+\Delta)u=F+G$. Let $n:I\to\R^+$ and $\lambda:I\to\R^+$ satisfy $\lambda(t)>n(t)^{-1}$. Then
\begin{align*}
\int_I \sup_{x\in\R^4} \int_{|x-y|\leq\lambda(t)} |u(t,y)|^2\,dy\,dt & \lesssim \|\lambda\|_{L_t^\infty}^2\bigl[ \|u\|_{L_t^\infty L_x^2}^2 +
	\|F\|_{L_t^2 L_x^{\frac43}}^2\bigr] \\
&\quad +\bigl[1+\|\log(\lambda n)\|_{L_t^\infty}\bigr] \|G\|_{L_t^2 L_x^1}^2 \\
&\quad + \|G\|_{L_t^\infty L_x^{\frac85}}^2\int_I n(t)^{-3}\,dt,
\end{align*}
where all space-time norms are over $I\times\R^4$. 
\end{proposition}

\begin{proof} Defining the weight $\omega=\omega(t,x,y)=e^{-\frac{| x-y|^2}{\lambda(t)^2}}$, it suffices to estimate
$$
\int_I \sup_{x\in\R^4} \|u(t)\|_{L^2(\omega\,dy)}^2\,dt.
$$ 
For this, we use the double Duhamel trick. That is, we write $u$ in the form
$$
u(t)=\sum_{j=1}^3 a_j(t)+b(t)=\sum_{j=1}^3 c_j(t)+d(t)
$$
and use the following inequality, which is a consequence of Cauchy--Schwarz:
$$
\| u(t)\|_{L^2(\omega\,dy)}^2\lesssim \sum_{j=1}^3 \|a_j(t)\|_{L^2(\omega\,dy)}^2+\sum_{j=1}^3 \|  c_j(t)\|_{L^2(\omega\,dy)}^2+|\langle   b(t),  d(t)\rangle_{L^2(\omega\,dy)}|.
$$

With $I=(t_0,t_1)$, we choose our decomposition of $u$ as follows: first,
\begin{align*}
&a_1(t)=e^{i(t-t_0)\Delta}u(t_0)-i\int_{t_0}^t e^{i(t-s)\Delta}F(s)\,ds,\\
&a_2(t)=-i\int_{t_0}^{t-\lambda(t)^{2}} e^{i(t-s)\Delta} G(s)\,ds, \quad a_3(t)=-i\int_{t-n(t)^{-2}}^t e^{i(t-s)\Delta} G(s)\,ds, \\
&b(t)=-i\int_{t-\lambda(t)^{2}}^{t-n(t)^{-2}} e^{i(t-s)\Delta} G(s)\,ds, \\
\end{align*}
and similarly
\begin{align*}
&c_1(t)=e^{-i(t_1-t)\Delta}u(t_1)+i\int_t^{t_1} e^{-i(\tau-t)\Delta}F(\tau)\,d\tau, \\
&c_2(t)=i\int_{t+\lambda(t)^{2}}^{t_1} e^{-i(\tau-t)\Delta}G(\tau)\,d\tau, \quad c_3(t)=i\int_t^{t+n(t)^{-2}} e^{-i(\tau-t)\Delta}G(\tau)\,d\tau, \\
&d(t)=i\int_{t+n(t)^{-2}}^{t+\lambda(t)^{2}}e^{-i(\tau-t)\Delta}G(\tau)\,d\tau.
\end{align*}

We first estimate
\begin{align*}
\|  a_1(t)\|_{L^2(\omega\,dy)}^2& \lesssim \|e^{\frac{-|\cdot|^2}{\lambda(t)^2}}\|_{L_y^2}\|  a_1(t)\|_{L_y^4}^2\lesssim\lambda(t)^{2}
	\|a_1(t)\|_{L_y^4}^2
\end{align*}
uniformly in $x$, so that by Strichartz we have
$$
\int_I\sup_{x\in\R^4} \|  a_1(t)\|_{L^2(\omega\,dy)}^2\,dt\lesssim\|\lambda\|_{L_t^\infty}^2\|a_1\|_{L_t^2 L_x^4}^2
	\lesssim\|\lambda\|_{L_t^\infty}^2\bigl[\|u\|_{L_t^\infty L_x^2}^2+\|F\|_{L_t^2 L_x^{\frac{4}{3}}}^2\bigr].
$$
We can estimate the contribution of $c_1$ in the same way.

Next, we estimate the long-time piece $a_2$. We have
$$
\|  a_2(t)\|_{L^2(\omega\,dy)}^2\lesssim\|e^{\frac{-| \cdot|^2}{\lambda(t)^2}}\|_{L_y^1}\|  a_2(t)\|_{L_y^\infty}^2
	\lesssim \lambda(t)^4\|a_2(t)\|_{L_y^\infty}^2
$$
uniformly in $x$. We now use the dispersive estimate to estimate
\begin{align*}
\|a_2(t)\|_{L_y^\infty}&\lesssim\sum_{j=0}^{\log [(t-t_0)\lambda(t)^{-2}]}\int_{t-s\sim 2^j \lambda(t)^{2}}| t-s|^{-2}\|G(s)\|_{L_y^1}\,ds\\
 &\lesssim \lambda(t)^{-2}\sum_{j=0}^\infty 2^{-j}\M(\|G(\cdot)\|_{L_y^1})(t)\lesssim \lambda(t)^{-2}\M(\|G(\cdot)\|_{L_y^1})(t).
\end{align*}
Thus 
	$$\int_I \sup_{x\in\R^4} \|  a_2(t)\|_{L^2(\omega\,dy)}^2\,dy\lesssim\|G\|_{L_t^2 L_x^1}^2.$$
We can estimate $c_2$ similarly. 

For the short-time piece $a_3$, we use Strichartz and H\"older's inequality to estimate
\begin{align*}
\|  a_3(t)\|_{L^2(\omega\,dy)}^2\lesssim \|e^{\frac{-|\cdot|^2}{\lambda(t)^2}}\|_{L_y^\infty}\|a_3(t)\|_{L_y^2}^2
 &\lesssim \|G\|_{L_s^{\frac{4}{3}} L_x^{\frac{8}{5}}((t-n(t)^{-2},t)\times\R^4)}^2
\\ & \lesssim n(t)^{-3}\|G\|_{L_t^\infty L_x^{\frac{8}{5}}}^2
\end{align*}
uniformly in $x$, so that
$$
\int_I \sup_{x\in\R^4} \| a_3(t)\|_{L^2(\omega\,dy)}^2\,dy\lesssim_u \|G\|_{L_t^\infty L_x^{\frac{8}{5}}}^2\int_I n(t)^{-3}\,dt.
$$
We can estimate $c_3$ similarly.

We now turn to the inner product term. We recall the following estimate from \cite[Proposition~2.7]{MMZ},  which follows from the evaluation of some Gaussian integrals:
$$
\sup_{x\in\R^4}\|e^{it\Delta}e^{-\frac{|x-y|^2}{\lambda^2}}e^{is\Delta}\|_{L_y^1\to L_y^\infty}\lesssim (s+t)^{-2}.
$$
Thus we can estimate
\begin{align*}
|\langle & b(t),d(t)\rangle_{L^2(\omega\,dy)}| \\
&=\bigg|\int_{t-\lambda(t)^2}^{t-n(t)^{-2}}\int_{t+n(t)^{-2}}^{t+\lambda(t)^2}
	\int  \bar{G}(\tau,y)e^{i(\tau-t)\Delta}e^{-\frac{|x-y|^2}{\lambda^2}}e^{i(t-s)\Delta}  G(s,y)\,dy\,d\tau\,ds\bigg| \\ 
&\lesssim \int_{t-\lambda(t)^2}^{t-n(t)^{-2}}\int_{t+n(t)^{-2}}^{t+\lambda(t)^2}(\tau-s)^{-2}\|  G(s)\|_{L_y^1}\|  G(\tau)\|_{L_y^1}\,d\tau\,ds \\ 
&\lesssim\sum_{\log[n(t)^{-2}]\leq j\leq k\leq\log[\lambda(t)^2]}2^{-2k}\int_{\tau-t\sim 2^k}\|G(\tau)\|_{L_y^1}\,d\tau
	\int_{t-s\sim 2^j}\|G(s)\|_{L_y^1}\,ds \\
&\lesssim \log[n(t)\lambda(t)]\,|\M(\|G(\cdot)\|_{L_y^1})(t)|^2,
	\end{align*}
uniformly in $x$. Thus
$$
\int_I \sup_{x\in\R^4} |\langle   b(t),  d(t)\rangle_{L^2(\omega\,dy)}|\,dt\lesssim\|\log(n\lambda)\|_{L_t^\infty} \|G\|_{L_t^2 L_x^1}^2.
$$
Collecting the estimates above, we complete the proof. \end{proof} 

\section{Long-time Strichartz estimates}\label{S:lts}

In this section we prove a long-time Strichartz estimate for almost periodic solutions to \eqref{nls}. Such estimates first appeared in the work of Dodson \cite{Dodson3}, and have since appeared in \cite{Dodson14, KV3D, MMZ, Mu, Mu3, Visan2011}. As in \cite{Dodson14, KV3D}, the long-time Strichartz estimate we prove relies on the maximal Strichartz estimate (Proposition~\ref{P:maximal}). We use these estimates in Section~\ref{S:rfc}, in which we rule out rapid frequency-cascades, as well as in Section~\ref{S:QS}, in which we prove an interaction Morawetz estimate to rule out quasi-solitons. 

For $u:I\times\R^4\to\C$ an almost periodic solution to \eqref{nls} as in Theorem~\ref{T:AP} and $4<q\leq\infty$, let
\begin{align}
&A(N):=\big(\sum_{M\leq N}\|\nsc u_M\|_{L_t^2 L_x^4(I\times\R^4)}^2\big)^{\frac{1}{2}},	\label{def:A}
\\
&B_q(N):=N^{\frac{5}{2}}\|\sup_{M>N}M^{\frac4q-2}\|u_M(t)\|_{L_x^q(\R^4)}\|_{L_t^2(I)},	\label{def:B}
\\
&K:=\int_I N(t)^{-3}\,dt\sim_u \sum_{J_k\subset I} N_k^{-5}.	\label{def:K}
\end{align}

The main result of this section is the following proposition. 

\begin{proposition}[Long-time Strichartz estimate]\label{prop:lts}
Let $u$ be an almost periodic solution as in Theorem~\ref{T:AP}. Let $I\subset[0,T_{max})$ be a compact time interval, which is a contiguous union of characteristic subintervals $J_k$. For any $N>0$ and $4<q\leq\infty$, 
\begin{equation}\label{eq:lts}
A(N)+B_q(N)\lesssim_u 1+N^{\frac{5}{2}}K^{\frac{1}{2}}.
\end{equation}
Furthermore, the implicit constant does not depend on $I$. 
\end{proposition}

\begin{remark} The proof that we give requires $4<q<8$; cf. \eqref{eq:lts hang1}. As Bernstein implies $B_{q}(N)\lesssim B_{r}(N)$ for $q>r$, we can deduce the result for $8\leq q\leq \infty$ \emph{a posteriori}. 
\end{remark}

The proof is by induction. The inductive step relies on the following lemma.
\begin{lemma}\label{lemma:lts} Let $0<\eta\ll 1$ and $4<q<8$. For any $N>0$ we have
\begin{equation}\label{eq:recurrence}
A(N)+B_q(N)\lesssim_u 1+C(\eta) N^{\frac{5}{2}}K^{\frac{1}{2}}+\eta [A(2N)+B_q(2N)].
\end{equation}
\end{lemma}

\begin{proof} We take space-time norms are over $I\times\R^4$ unless stated otherwise.

To begin, note that for any decomposition $|u|^4 u = F+G$, we may apply the standard Strichartz estimate (Proposition~\ref{P:Strichartz}) to $u_{\leq N}$, the maximal Strichartz estimate (Proposition~\ref{P:maximal}) to $u_{>N}$, Bernstein, and \eqref{E:bdd} to deduce
$$
A(N)+B_q(N)\lesssim_u 1+\|\nsc F\|_{L_t^2 L_x^{\frac{4}{3}}}+N^{\frac{5}{2}}\|G\|_{L_t^2 L_x^1}.
$$ 

We choose $c=c(\eta)$ as in \eqref{APsmall} and write $| u|^4u = F+G$, with
$$
F=\text{\O}(u_{\leq cN(t)}^2 u_{\leq 2N}^2 u),\quad G=\text{\O}(u_{\leq cN(t)}^2u_{>2N}^2u)+\text{\O}(u_{>cN(t)}^2u^3).
$$

Using fractional calculus, Sobolev embedding, Lemma~\ref{lem:endpt}, \eqref{E:bdd}, and \eqref{APsmall}, we first estimate
\begin{align}
\|\nsc(&u_{\leq cN(t)}^2 u_{\leq 2N}^2 u)\|_{L_t^2 L_x^{\frac{4}{3}}} \nonumber\\ 
&\lesssim \|u_{\leq cN(t)}\|_{L_t^\infty L_x^8}\|u\|_{L_t^\infty L_x^8}^3\|\nsc u_{\leq 2N}\|_{L_t^2 L_x^4}\nonumber\\ 
&\quad+\|u_{\leq cN(t)}\|_{L_t^\infty L_x^8}\|u\|_{L_t^\infty L_x^8}\|\nsc u\|_{L_t^\infty L_x^2}
	\|u_{\leq 2N}\|_{L_t^4L_x^\infty}^2\nonumber\\ 
&\lesssim_u \eta A(2N).	\label{E:LTShang}
\end{align}

Next, as $2<q<8$ we have $\tfrac{4}{q}-2<0<\tfrac{4}{q}-\tfrac{1}{2}.$ Thus, defining
$$
S=\{M_1,M_2, M_3\,|\, M_1\geq M_2 \geq M_3,\ M_2>2N\}
$$ 
and using Bernstein, Cauchy--Schwarz, and \eqref{E:bdd}, we find
\begin{align}
\|&u_{\leq cN(t)}^2 u_{>2N}^2 u\|_{L_t^2 L_x^1}\nonumber \\
& \lesssim \|u_{\leq cN(t)}\|_{L_t^\infty L_x^8}^2
	\|u_{>2N}^2 u\|_{L_t^2 L_x^{\frac{4}{3}}}\nonumber \\
&\lesssim\eta^2\bigg\|\sum_S\|u_{M_1}(t)\|_{L_x^2}\|u_{M_2}(t)\|_{L_x^q}
	\|u_{M_3}(t)\|_{L_x^\frac{4q}{q-4}}\bigg\|_{L_t^2}\nonumber\\ 
&\lesssim \eta^2\bigg\|\sum_S M_1^{-\frac32}M_3^{\frac4q-\frac12}\| u_{M_1}(t)\|_{\dot{H}_x^{\frac32}}\|u_{M_2}(t)\|_{L_x^q}
	\|u_{M_3}(t)\|_{\dot{H}_x^{\frac32}}\bigg\|_{L_t^2}\nonumber\\ 
&\lesssim \eta^2 \bigg\|\sup_{M>2N} M^{\frac4q-2}\|u_M(t)\|_{L_x^q}\sum_{M_1\geq M_3} (\tfrac{M_3}{M_1})^{\frac{4}{q}-\frac{1}{2}-}
	\|u_{M_1}(t)\|_{\dot{H}_x^{\frac32}}\|u_{M_3}(t)\|_{\dot{H}_x^{\frac32}}\bigg\|_{L_t^2}\nonumber\\ 
&\lesssim \eta^2 \|\nsc u\|_{L_t^\infty L_x^2}^2N^{-\frac{5}{2}}B_q(2N)\lesssim_u \eta^2 N^{-\frac{5}{2}} B_q(2N). \label{eq:lts hang1}
\end{align}	
	
Finally, restricting attention to an individual characteristic subinterval $J_k$, we use Sobolev embedding, Bernstein, and \eqref{E:bdd} to estimate
\begin{align*}
\|u_{>cN_k}^2 u^3\|_{L_t^2 L_x^1}
\lesssim \|u_{>cN_k}\|_{L_t^4 L_x^{\frac{16}{5}}}^2\|u\|_{L_t^\infty L_x^8}^3
&\lesssim_u \||\nabla|^{\frac{1}{4}}u_{>cN_k}\|_{L_t^4L_x^{\frac{8}{3}}}^2\\ 
&\lesssim_u C(\eta) N_k^{-\frac{5}{2}}\|\nsc u\|_{L_t^4 L_x^{\frac{8}{3}}}^2.
\end{align*}
We now square and sum over $J_k\subset I$, using Lemma~\ref{L:SB} and \eqref{def:K}. We find
\begin{equation}\label{eq:lts hang2}
\|u_{>cN(t)}^2u^3\|_{L_t^2L_x^1}\lesssim_u C(\eta) K^{\frac{1}{2}}.
\end{equation}

Collecting our estimates, we complete the proof of Lemma~\ref{lemma:lts}. 
\end{proof}

With Lemma~\ref{lemma:lts} in place we can now prove Proposition~\ref{prop:lts}.

\begin{proof}[Proof of Proposition~\ref{prop:lts}] We proceed by induction. For the base case, take $N\geq\sup_{t\in I}N(t)$. In this case, we first use Strichartz (Proposition~\ref{P:Strichartz}) and Lemma~\ref{L:SB} to estimate
\begin{align*}
[A(N)]^2\lesssim_u1+\int_I N(t)^2\lesssim_u 1+N^5K.
\end{align*}
For $B_q(N)$, we instead use the maximal Strichartz estimate (Proposition~\ref{P:maximal}), Lemma~\ref{L:SB}, fractional calculus, and \eqref{E:bdd}. We find
\begin{align*}
B_q(N)&\lesssim_u N^{\frac{5}{2}}\bigl[ \||\nabla|^{-1}u_{>N}\|_{L_t^\infty L_x^2}+
	\||\nabla|^{-1}P_{>N}\bigl(|u|^4 u\bigr)\|_{L_t^2L_x^{\frac{4}{3}}}\bigr]\\ 
&\lesssim_u 1+\|u\|_{L_t^\infty L_x^8}^4\|\nsc u\|_{L_t^2 L_x^4}\lesssim_u 1+N^{\frac{5}{2}}K^{\frac{1}{2}}.
\end{align*}

If we now suppose that \eqref{eq:lts} holds at frequency $2N$, then we can use Lemma~\ref{lemma:lts} to show that \eqref{eq:lts} holds at frequency $N$, provided we choose $\eta=\eta(u)$ sufficiently small. For the details of such an argument, one can refer to \cite{MMZ, Mu, Mu3, Visan2011}.
\end{proof} 

We record here some consequences of Proposition~\ref{prop:lts} to be used in Section~\ref{S:QS}.
\begin{corollary}\label{C:lts} Let $u,I$ be as in Proposition~\ref{prop:lts}, with $K$ as in \eqref{def:K}. Define
\begin{equation}\label{hivslo}
\PHi=P_{>K^{-\frac15}},\quad \PLo=P_{\leq K^{-\frac15}},\quad \uhi=\PHi u,\quad \ulo=\PLo u.
\end{equation}
Then
\begin{align}
&\|\nsc\ulo\|_{L_{t,x}^3}+\|\nsc\ulo\|_{L_t^4 L_x^{\frac83}}+\|\ulo\|_{L_t^4 L_x^\infty}\lesssim_u 1,\label{E:ulo} \\
&\|\sup_M M^{-2}\|P_M\uhi\|_{L_x^\infty}\|_{L_t^2}^2\lesssim_u K, \label{E:uhi1} \\
&\|\nabla\uhi\|_{L_t^{12} L_x^{\frac{24}{11}}} \lesssim_u K^{\frac{1}{10}}, \label{E:uhi2} \\
&\|\PHi(\ulo^2 u^3)\|_{L_t^2 L_x^{\frac43}}^2\lesssim_u K^{\frac35}, 
\quad \|\uhi^2 u^3\|_{L_t^2 L_x^1}^2\lesssim_u K,\label{E:non}
\end{align}
where all space-time norms are over $I\times\R^4$. 
\end{corollary}

\begin{proof} By Proposition~\ref{prop:lts} and the definition of $\PLo$, 
\begin{equation}\label{pf:ulo}
\|\nsc\ulo\|_{L_t^2 L_x^4}^2\lesssim \sum_{M}\|\nsc P_M \ulo\|_{L_t^2 L_x^4}^2\lesssim_u 1.
\end{equation}
Thus \eqref{E:ulo} follows from interpolation, \eqref{pf:ulo}, \eqref{E:bdd}, and Lemma~\ref{lem:endpt}. The estimate \eqref{E:uhi1} also follows immediately from Proposition~\ref{prop:lts} (with $q=\infty$) and the definition of $\PHi$. 

Next, by interpolation and Proposition~\ref{prop:lts}, 
\begin{align*}
\|\nabla\uhi\|_{L_t^{12} L_x^{\frac{24}{11}}}^6&\lesssim \| |\nabla|^{-\frac32}\uhi\|_{L_t^2 L_x^4}
	\| \nsc \uhi\|_{L_t^\infty L_x^2}^{5} \lesssim_u\sum_{N>K^{-\frac15}} \||\nabla|^{-\frac32} u_N\|_{L_t^2 L_x^4} \\
&\lesssim_u \sum_{N>K^{-\frac15}} N^{-3}(1+N^5 K)^{\frac12} \lesssim_u K^{\frac{3}{5}},
\end{align*}
which implies \eqref{E:uhi2}. 

For the first estimate in \eqref{E:non}, we use Bernstein, estimate as in \eqref{E:LTShang}, and use Proposition~\ref{prop:lts}. For the second estimate in \eqref{E:non}, we estimate as in \eqref{eq:lts hang1} and use Proposition~\ref{prop:lts}. \end{proof}


\section{Preclusion of rapid frequency-cascades}\label{S:rfc}

In this section, we preclude the existence of almost periodic solutions as in Theorem~\ref{T:AP} for which \eqref{E:KF} holds. Throughout this section, we denote
$$
K=\int_0^{T_{max}}N(t)^{-3}\,dt<\infty.
$$

We will use the long-time Strichartz estimate (Proposition~\ref{prop:lts}) and the reduced Duhamel formula (Proposition~\ref{P:RD}) to show that the existence of such solutions is inconsistent with the conservation of mass. 

Note that \eqref{E:KF} implies
\begin{equation}\label{eq:cascade}
\lim_{t\to T_{max}}N(t)=\infty.
\end{equation}
This is clear if $T_{max}=\infty$, while if $T_{max}<\infty$, this follows from Corollary~\ref{cor:blowup}. 

\begin{theorem}[No rapid frequency-cascades] There are no almost periodic solutions as in Theorem~\ref{T:AP} such that \eqref{E:KF} holds.
\end{theorem}

\begin{proof} As $K$ is finite, we can extend the conclusion of Proposition~\ref{prop:lts} to the whole interval $[0,T_{max})$. Throughout the proof, all space-time norms are taken over $[0,T_{max})\times\R^4$. We proceed in three steps.

\textbf{Step 1.} We show that
\begin{equation}\label{eq:rf mass1}
\|u_{N<\cdot\leq 1}\|_{L_t^\infty L_x^2}+N^{-\frac32}A(N)\lesssim_u 1
\end{equation}
uniformly for $0<N<1$, where $A(N)$ is as in \eqref{def:A}.

By \eqref{E:bdd}, Bernstein, Proposition~\ref{prop:lts}, and \eqref{E:KF}, the quantity appearing in \eqref{eq:rf mass1} is finite for each $N>0$: 
\begin{align}
\|u_{N<\cdot\leq 1}\|_{L_t^\infty L_x^2}+N^{-\frac32}A(N)&\lesssim_u N^{-\frac32}\big[1+N^{\frac{5}{2}}K^{\frac{1}{2}}\big]
<\infty.\label{eq:rf mass2}
\end{align}

To begin, the reduced Duhamel formula (Proposition~\ref{P:RD}) and Strichartz imply
\begin{align*}
\text{LHS}\eqref{eq:rf mass1}&\lesssim\|P_{N<\cdot\leq 1}\bigl(|u|^4u\bigr)\|_{L_t^2 L_x^{\frac{4}{3}}}
	+N^{-\frac32}\|\nsc P_{\leq N}\bigl(|u|^4u\bigr)\|_{L_t^2 L_x^{\frac{4}{3}}}.
\end{align*} 

Now we let $\eta>0$ and choose $c=c(\eta)$ as in \eqref{APsmall}. To decompose the nonlinearity, we first write $u=u_{\leq cN(t)}+u_{>cN(t)}$ and subsequently $u=u_{\leq N}+u_{N<\cdot\leq 1}+u_{>1}$. According to our notation,  $u_{\leq N}$ and $ u_{N<\cdot\leq 1}$ are both $\text{\O}(u_{\leq 1})$. Thus,
\begin{align*}
|u|^4u&=\text{\O}(u^3u_{>cN(t)}^2)+\text{\O}(u^3u_{>1}^2)+\text{\O}(u^2u_{\leq cN(t)}u_{\leq N}^2)
	+\text{\O}(u u_{\leq cN(t)} u_{\leq 1}^2 u_{N<\cdot\leq 1}).
\end{align*}

First, Bernstein, \eqref{eq:lts hang2}, and \eqref{E:KF} imply
\begin{align*}
\|P_{N<\cdot\leq 1}&(u^3u_{>cN(t)}^2)\|_{L_t^2 L_x^{\frac{4}{3}}}+
	N^{-\frac32}\||\nabla|^{\frac32}P_{\leq N}(u^3u_{>cN(t)}^2)\|_{L_t^2 L_x^{\frac{4}{3}}} \\
&\lesssim (1+N)\|u^3u_{>cN(t)}^2\|_{L_t^2 L_x^1} \lesssim_u C(\eta) K^{\frac{1}{2}}\lesssim_u 1.
\end{align*}

Second, using Bernstein, \eqref{APsmall}, \eqref{eq:lts hang1}, Proposition~\ref{prop:lts}, and \eqref{E:KF},  we have
\begin{align*}
\|P_{N<\cdot\leq 1}&(u^3 u_{>1}^2)\|_{L_t^2 L_x^{\frac{4}{3}}}+N^{-\frac32}\|\nsc P_{\leq N}(u^3u_{>1}^2)\|_{L_t^4 L_x^{\frac{4}{3}}}\\
& \lesssim (1+N)\|u^3u_{>1}^2 \|_{L_t^2 L_x^1}\lesssim \|u\|_{L_t^\infty L_x^8}^2\|u_{>1}^2u\|_{L_t^2 L_x^{\frac{4}{3}}}
	\lesssim_u 1+K^{\frac{1}{2}}\lesssim_u 1.
\end{align*}
	
Third, Bernstein, fractional calculus, Lemma~\ref{lem:endpt}, \eqref{APsmall} and \eqref{E:bdd} give
\begin{align*}	
\|P_{N<\cdot\leq 1}&(u^2 u_{\leq cN(t)} u_{\leq N}^2)\|_{L_t^2L_x^{\frac{4}{3}}}
	+N^{-\frac32}\||\nabla|^{\frac32}P_{\leq N}(u^2 u_{\leq cN(t)}u_{\leq N}^2)\|_{L_t^2 L_x^{\frac{4}{3}}} \\
&\lesssim N^{-\frac32}\|u\|_{L_t^\infty L_x^8}\|\nsc u\|_{L_t^\infty L_x^2}\|u_{\leq cN(t)}\|_{L_t^\infty L_x^8}
	\|u_{\leq N}\|_{L_t^4 L_x^\infty}^2 \\ 
&\quad+N^{-\frac32}\|u\|_{L_t^\infty L_x^8}^2\|\nsc u_{\leq cN(t)}\|_{L_t^\infty L_x^2}\|u_{\leq N}\|_{L_t^4L_x^\infty}^2\\ 
&\quad+N^{-\frac32}\|u\|_{L_t^\infty L_x^8}^3\|u_{\leq cN(t)}\|_{L_t^\infty L_x^8}\|\nsc u_{\leq N}\|_{L_t^2 L_x^4} \lesssim_u \eta N^{-\frac32}A(N).
\end{align*}

Finally, by Bernstein, Lemma~\ref{lem:endpt}, \eqref{E:bdd}, \eqref{APsmall}, Proposition~\ref{prop:lts}, and \eqref{E:KF},
\begin{align*}
\|u &u_{\leq cN(t)}u_{\leq 1}^2 u_{N<\cdot\leq 1}\|_{L_t^2L_x^{\frac{4}{3}}}+N^{-\frac32}
	\|\nsc P_{\leq N}(uu_{\leq cN(t)}u_{\leq 1}^2 u_{N<\cdot\leq 1})\|_{L_t^2 L_x^{\frac{4}{3}}} \\ 
&\lesssim \|u\|_{L_t^\infty L_x^8}\|u_{\leq cN(t)}\|_{L_t^\infty L_x^8}\|u_{\leq 1}\|_{L_t^4L_x^\infty}^2\|u_{N<\cdot\leq 1}\|_{L_t^\infty L_x^2}\\ 
&\lesssim_u \eta(1+K^{\frac{1}{2}})\|u_{N<\cdot\leq 1}\|_{L_t^\infty L_x^2}\lesssim_u \eta \|u_{N<\cdot\leq 1}\|_{L_t^\infty L_x^2}.
\end{align*}

Collecting our estimates, we find
$$
\text{LHS}\eqref{eq:rf mass1}\lesssim_u 1+\eta\text{LHS}\,\eqref{eq:rf mass1}.
$$
Choosing $\eta$ sufficiently small and recalling \eqref{eq:rf mass2}, we recover \eqref{eq:rf mass1}. 

We record here an important consequence of \eqref{eq:rf mass1}, namely
\begin{equation}\label{eq:rf mass3}
\|u\|_{L_t^\infty L_x^2}\lesssim_u 1.
\end{equation}
Indeed, sending $N\to 0$ in \eqref{eq:rf mass1} one can control the low frequencies, while Bernstein and \eqref{E:bdd} control the high frequencies.  

\textbf{Step 2.} We upgrade \eqref{eq:rf mass3} to  
\begin{equation}\label{eq:rf mass4}
\|u_{\leq N}\|_{L_t^\infty L_x^2}\lesssim_u N\quad\text{for}\quad N>0.
\end{equation}
First, the reduced Duhamel formula (Proposition~\ref{P:RD}) and Bernstein imply
\begin{align*}
\|u_{\leq N}\|_{L_t^\infty L_x^2}&\lesssim\|P_{\leq N}\bigl(|u|^4u\bigr)\|_{L_t^2 L_x^{\frac{4}{3}}}\lesssim N\||u|^4u\|_{L_t^2 L_x^1}.
\end{align*} 

Writing $|u|^4u=\text{\O}(u_{\leq 1}^3u^2)+\text{\O}(u_{>1}^2 u^3)$, we then use Lemma~\ref{lem:endpt}, Bernstein, \eqref{E:bdd}, \eqref{eq:rf mass3}, Proposition~\ref{prop:lts}, \eqref{E:KF} and \eqref{eq:lts hang1} to estimate
\begin{align*}
&\|u_{\leq 1}^3u^2\|_{L_t^2 L_x^1}\lesssim \|u_{\leq 1}\|_{L_t^4 L_x^\infty}^2\|u_{\leq 1}\|_{L_{t,x}^\infty}
	\|u\|_{L_t^\infty L_x^2}^2\lesssim_u 1,\\
&\|u_{>1}^2 u^3\|_{L_t^2 L_x^1}\lesssim_u 1.
\end{align*} 

\textbf{Step 3.} We show $\|u(t)\|_{L_x^2}\equiv 0$, contradicting that $u$ is a blowup solution. 

We first establish some negative regularity: Bernstein, \eqref{eq:rf mass4} and \eqref{E:bdd} imply
\begin{align}
\||\nabla|^{-\frac{1}{2}}u\|_{L_t^\infty L_x^2}&\lesssim_u \||\nabla|^{-\frac{1}{2}}u_{>1}\|_{L_t^\infty L_x^2}+
	\sum_{N\leq 1} \||\nabla|^{-\frac{1}{2}}u_N\|_{L_t^\infty L_x^2}\nonumber \\ 
&\lesssim_u \||\nabla|^{\frac32}u\|_{L_t^\infty L_x^2}+\sum_{N\leq 1} N^{\frac{1}{2}}\lesssim_u 1. \label{eq:rf negative}
\end{align}

Given $\eta>0$, we choose $c=c(\eta)$ as in \eqref{APsmall}. Interpolating \eqref{eq:rf negative} with \eqref{APsmall} yields
$$
\|P_{\leq cN(t)}u\|_{L_t^\infty L_x^2}\lesssim_u \eta^{\frac{1}{4}}.
$$ 
On the other hand, Bernstein and \eqref{E:bdd} give
$$
\|P_{>cN(t)}u\|_{L_t^\infty L_x^2}\lesssim [c(\eta)N(t)]^{-\frac32}\|\nsc u\|_{L_t^\infty L_x^2}\lesssim_u [c(\eta)N(t)]^{-\frac32}.
$$
Choosing $\eta$ small, sending $t\to T_{max}$, and recalling \eqref{eq:cascade}, we deduce that $\|u(t)\|_{L_x^2}\to 0$ as $t\to T_{max}$. By conservation of mass, it follows that $\|u(t)\|_{L_x^2}\equiv 0$, as needed. 
\end{proof}

\section{Preclusion of quasi-solitons}\label{S:QS}

In this section we preclude the possibility of almost periodic solutions to \eqref{nls} such that \eqref{E:KI} holds. Our main tool is a space-localized interaction Morawetz inequality (Proposition~\ref{P:IM}).  To control the error terms, we rely first on the long-time Strichartz estimate (specifically, Corollary~\ref{C:lts}). We also use Proposition~\ref{L:mass}, which suffers a logarithmic loss (cf. Corollary~\ref{C:mass} below). As explained in the introduction, we overcome this logarithmic loss by using an appropriate Morawetz weight and exploiting the energy-supercriticality of \eqref{nls}.

The main result of this section is the following theorem. 

\begin{theorem}[No quasi-solitons]\label{T:QS} There are no almost periodic solutions as in Theorem~\ref{T:AP} such that \eqref{E:KI} holds. 
\end{theorem}

Throughout this section, we suppose $u:[0,T_{max})\times\R^4\to\C$ is an almost periodic solution as in Theorem~\ref{T:AP} such that \eqref{E:KI} holds. We let $I\subset[0,T_{max})$ be a compact time interval, which is a contiguous union of characteristic subintervals, and we denote
\begin{equation}\label{E:QSK}
K:=\int_I N(t)^{-3}\,dt.
\end{equation}
By \eqref{E:KI}, we can make $K$ arbitrarily large by choosing $I$ sufficiently large inside $[0,T_{max})$. Our goal is to prove an interaction Morawetz inequality for $u$ on $I\times\R^4$ that we can use to contradict \eqref{E:KI}. Specifically, we will prove that for $I\subset[0,T_{max})$ sufficiently large and any $\eta>0$, we have $K\lesssim_u \eta K$. Choosing $\eta=\eta(u)$ sufficiently small will then yield the contradiction $K=0$.

As in \eqref{hivslo}, we define
\begin{equation}\label{def:hi}
\PHi=P_{>K^{-\frac15}},\quad \PLo=P_{\leq K^{-\frac15}},\quad \uhi=\PHi u,\quad \ulo=\PLo u.
\end{equation}


\subsection{Setup}\label{S:setup} Given a weight $a:I\times\R^4\to\R$ and $u$ as above, we define the interaction Morawetz action by
\begin{equation}\label{E:M}
M(a;t)=\iint_{\R^4\times\R^4}|u(t,y)|^2 a_k(t,x-y) 2\Im[\bar{u}(t,x)u_k(t,x)]\,dx\,dy,
\end{equation}
where subscripts denote spatial partial derivatives and repeated indices are summed. 

A standard computation yields the following identity:
\begin{proposition}[Interaction Morawetz identity] \label{P:identity}
\begin{align}
\tfrac{d}{dt}& M(a;t) \nonumber \\
&=\iint |u(y)|^2 \partial_t a_k(x-y)2\Im(\bar{u}u_k)(x)\,dx\,dy \label{E:dta} \\
&+ \iint 4a_{jk}(x-y)\bigl[|u(y)|^2\Re(\bar{u}_ku_j)(x)-\Im(\bar{u}u_j)(y)\Im(\bar{u}u_k)(x)\bigr]\,dx\,dy \label{E:scary}\\ 
&+ \iint |u(y)|^2 \tfrac{4}{3}\Delta a(x-y) |u(x)|^6\,dx\,dy \label{E:potential} \\
&+\iint |u(y)|^2 (-\Delta\Delta a)(x-y) |u(x)|^2\,dx\,dy, \label{E:massmass}
\end{align}
where here and below we suppress the dependence of functions on $t$. 
\end{proposition}

For the standard interaction Morawetz estimate, introduced originally in \cite{ckstt}, one takes $a(x)=|x|$. By proving lower bounds for $\frac{d}{dt}M$ and using the fundamental theorem of calculus, one can deduce (in dimensions four and higher) 
$$
\iiint \frac{|u(x)|^2 |u(y)|^2}{|x-y|^3}\,dx\,dy\,dt\lesssim \sup_t |M(a;t)|\lesssim \|u\|_{L_t^\infty L_x^2}^3 \|\nabla u\|_{L_t^\infty L_x^2}.
$$
This estimate is not directly applicable in our setting, since we do not control the $H_x^1$-norm of $u$. In order to make $M(a;t)$ finite, we choose our weight to be of the form
\begin{equation}\label{D:a}
a(t,x)=\tfrac{1}{n(t)} w(n(t)|x|),
\end{equation}
where $w$ is a truncation of $|x|$. The need to rescale $w$ by a function of $t$ stems from the logarithmic failure in Proposition~\ref{L:mass}. The most natural choice would be to rescale by the frequency scale function $N(t)$; however, \eqref{E:dta} would then involve $N'(t)$, over which we have no control.  Instead, we follow the approach of \cite{Dodson, Dodson14}, choosing $n(t)$ to be the output of a `smoothing algorithm' whose input is a function closely related to $N(t)$. 

The construction of the weight $w$ is motivated by \cite{KV3D}. We need a few parameters, which we later choose in terms of $K$. We let $R\gg 1$ and $J\sim\log R\gg 1$. We let $w$ be a smooth radial function, which we regard either as a function of $x$ or $r=|x|$, that satisfies the following: 
\begin{equation}\label{D:W}
w(0)=0, \quad w_r\geq 0,\quad 
w_r=\begin{cases} 1 & r\leq R \\
1-\tfrac{1}{J}\log(\tfrac{r}{R}) & Re<r\leq Re^{J-1} \\
0 & r> Re^J.
\end{cases}
\end{equation}
Thus $w=|x|$ for $|x|\leq R$ and $w$ is constant for $|x|>Re^J$. We fill in the regions where $w_r$ is not yet defined so that
\begin{equation}\label{E:higher}
|\partial_r^k w_r|\lesssim_k \tfrac{1}{J} r^{-k}
\end{equation}
for all $k\geq 1$, uniformly in all parameters.

As $\nabla a$ has compact support and $u\in L_t^\infty L_x^4$ (cf. Proposition~\ref{P:decay}), we have the following bound by Young's inequality and Sobolev embedding:
\begin{equation}\label{E:boundM}
|M(a;t)|\lesssim \|u(t)\|_{L_x^4}^2 \|\nabla a(t)\|_{L_x^1}\|u(t)\|_{L_x^8}\|\nabla u(t)\|_{L_x^{\frac83}}\lesssim_u
	 \bigl(\tfrac{Re^J}{n(t)}\bigr)^4.
\end{equation}

\subsection{Construction of $n(t)$}\label{S:nt} The construction of the function $n(t)$ is motivated by the work of Dodson \cite{Dodson, Dodson14}. We proceed inductively. To get started, we define 
\begin{equation}\label{D:n}
n_0(t):=\|\uhi(t)\|_{L_x^4}^{-2},
\end{equation}
where $u$ is as above and $\uhi$ is as in \eqref{def:hi}. Recall from Proposition~\ref{P:decay} that $u\in L_t^\infty L_x^4$. The motivation for this choice of $n_0$ stems from the estimation of \eqref{E:dta}, cf. \eqref{whym} below. 

We collect the key properties of $n_0$ in the following lemma. 
\begin{lemma}[Properties of $n_0$]\label{L:n0} For $I$ sufficiently large inside $[0,T_{max})$, 
\begin{align}\label{E:n01}
&1\lesssim_u n_0(t) \lesssim_u N(t)\quad\text{uniformly for }t\in I, \\
\label{E:n03}
&|n_0'(t)|\lesssim_u n_0(t)^3\quad\text{uniformly for}\quad t\in I, \\
\label{E:n02}
&\int_I n_0(t)^{-3}\,dt \lesssim_u K.
\end{align}

\end{lemma}

\begin{proof} As Proposition~\ref{P:decay} gives $u\in L_t^\infty L_x^4$, the first inequality in \eqref{E:n01} holds. For the second inequality, it suffices to show
\begin{equation}\label{E:n011}
\inf_{t\in I} N(t)^2 \|\uhi(t)\|_{L_x^4}^4\gtrsim_u 1.
\end{equation}
To this end, first note that as $u\not\equiv 0$, we may use almost periodicity, \eqref{E:Ng1}, and Sobolev embedding to deduce
\begin{equation}\nonumber
\inf_{t\in I}\|u(t)\|_{L_x^8}\gtrsim \inf_{t\in I} \|P_{\leq CN(t)}\uhi(t)\|_{L_x^8}\gtrsim_u 1
\end{equation}
for $C$ and $K$ sufficiently large depending on $u$.  In particular, recalling \eqref{E:KI}, this lower bound holds for $I$ sufficiently large inside $[0,T_{max})$. As Bernstein implies
$$
\|P_{\leq CN(t)}\uhi(t)\|_{L_x^8}^8\lesssim \|u_{\leq CN(t)}\|_{L_x^\infty}^4\|\uhi(t)\|_{L_x^4}^4
	\lesssim_u N(t)^{2}\|u(t)\|_{L_x^8}^4\|\uhi(t)\|_{L_x^4}^4,
$$
we deduce \eqref{E:n011}.

Next, using \eqref{nls}, integrating by parts, and using Sobolev embedding and \eqref{E:bdd}, we can estimate
$$
\bigl|\tfrac{d}{dt}\|\uhi(t)\|_{L_x^4}^4\bigr| \leq \|\nabla\uhi(t)\|_{L_x^{\frac83}}^2\|\uhi(t)\|_{L_x^8}^2+\|u(t)\|_{L_x^8}^8\lesssim_u 1,
$$
from which \eqref{E:n03} follows.

Finally, for \eqref{E:n02}, we define
$$
S=\{N_1,\dots, N_4\,|\, K^{-\frac15}\leq N_1\dots\leq N_4\}
$$
and use Bernstein, Cauchy--Schwarz, and Corollary~\ref{C:lts} to estimate
\begin{align*}
\int_I& n_0(t)^{-3}\,dt \\
&\lesssim \int_I \biggl[ \sum_S \|u_{N_1}\|_{L_x^\infty} \|u_{N_2}\|_{L_x^\infty} \|u_{N_3}\|_{L_x^2}
	\|u_{N_4}\|_{L_x^2}\biggr]^{\frac32}\,dt \\ 
&\lesssim \bigl\|\sup_M M^{-2}\|P_M\uhi\|_{L_x^\infty}\bigr\|_{L_t^2}^2 \|u\|_{L_t^\infty\dot H_x^{\frac32}} \biggl\| \sum_S 
	\bigl(\tfrac{N_1N_2}{N_3N_4}\bigr)^{\frac32} \|u_{N_3}\|_{\dot H_x^{\frac32}} \|u_{N_4}\|_{\dot H_x^{\frac32}}\biggr\|_{L_t^\infty}^{\frac32} \\
&\lesssim_u \|\nsc u\|_{L_t^\infty L_x^{2}}^4 K\lesssim_u K.
\end{align*}
This completes the proof of Lemma~\ref{L:n0}. \end{proof}

From $n_0$ we now construct a closely related function $n_1$, which is piecewise-linear and hence simpler to work with. First, using \eqref{E:n03}, we take $\delta=\delta(u)$ sufficiently small so that that 
$$
\tfrac12 n_0(t)\leq n_0(t')\leq 2n_0(t)\quad\text{whenever}\quad |t-t'|\leq \delta n_0(t)^{-2}.
$$
We now partition $I$ into intervals $J_\ell=[t_\ell,t_{\ell+1}]\cap I$, where $t_0=\inf I$ and $t_{\ell+1}=t_{\ell}+\delta n_0(t_\ell)^{-2}$. For each $J_\ell$, define
$$
k_\ell=\sup\{k\in\Z:\inf_{t\in J_\ell} n_0(t)\geq 2^k\}.
$$
By construction, the value of $k_\ell$ can change by at most 1 on adjacent intervals: 
$$
\frac{2^{k_\ell}}{2^{k_{\ell+1}}}\in\{\tfrac12,1,2\}.
$$

We now define $n_1(t_\ell)=2^{k_\ell}$, and we take $n_1$ to be the linear interpolation between the $t_\ell$. Between the final $t_\ell$ and $\sup I$ we take $n_1$ to be constant. By construction and Lemma~\ref{L:n0}, $n_1$ has the following properties.
\begin{lemma}[Properties of $n_1$] For $I$ sufficiently large inside $[0,T_{max})$, 
\begin{align}
\label{n101}
&n_1(t)\sim n_0(t)\quad\text{uniformly for}\quad t\in I, \\
\label{n102}
&|n_1'(t)|\lesssim_u n_1(t)^3\quad\text{uniformly for}\quad t\in I, \\
\label{n103}
&\int_I \frac{|n_1'(t)|}{n_1(t)^6}\,dt\lesssim_u K.
\end{align}
\end{lemma}

We consider the quantity in \eqref{n103} because it shows up when we estimate the error term \eqref{E:dta}; cf. \eqref{whym} below. We now describe an algorithm as in the work of Dodson \cite{Dodson, Dodson14}, which takes $n_1$ as input and generates a sequence of functions $n_m$. The algorithm increases the pointwise value of $n_m$, but decreases the quantity in \eqref{n103}. We discuss this tradeoff in more detail below; cf. \eqref{radius}. 

\begin{definition}[Smoothing algorithm, \cite{Dodson, Dodson14}] Call $t_\ell$ a \emph{low point} if there exist $m_1,m_2\geq 1$ such that
\begin{itemize}
\item $n_1(t_\ell)=n_1(t_{\ell-m})$ for $0\leq m<m_1$, and $n_1(t_\ell)<n_1(t_{\ell-m_1})$,
\item $n_1(t_\ell)=n_1(t_{\ell+m})$ for $0\leq m<m_2$, and $n_1(t_\ell)<n_1(t_{\ell+m_2}).$
\end{itemize} We can define \emph{high points} analogously.

If $t_k$ and $t_m$ are not themselves low points, but $t_\ell$ is a low point for all $k<\ell<m$, then we call $[t_{k+1},t_{m-1}]$ a \emph{valley}. Note that a valley may consist of a single point. Also note that by construction, $n_1(t_k)=n_1(t_m)=2n_1(t)$ for $t\in[t_{k+1},t_{m-1}].$ 

Similarly, if $t_k$ and $t_m$ are not themselves high points, but $t_\ell$ is a high point for all $k<\ell<m$, then we call $[t_{k+1},t_{m-1}]$ a \emph{peak}.

Note that peaks and valleys must alternate. If an interval $J$ joins a peak to a valley (or vice versa), we call $J$ a \emph{slope}. Note that $n_1$ is monotone on slopes.

We construct $n_2$ from $n_1$ by `filling in the valleys'. That is, if $[t_{k+1},t_{m-1}]$ is a valley, we set $n_2(t)=n_1(t_k)$ for $t\in(t_k,t_m)$. For all other points, we set $n_2(t)=n_1(t)$. We can similarly construct $n_3(t)$ from $n_2(t)$, and so on. This generates a sequence of functions $n_m(t)$. 
\end{definition}

The functions $n_m$ generated by the algorithm have the following properties. 

\begin{lemma}[Properties of $n_m$] For $I$ sufficiently large inside $[0,T_{max})$, 
\begin{align}\label{E:nm01}
&1\lesssim_u n_0(t)\lesssim_u n_m(t)\lesssim_u 2^{m} n_0(t)\quad\text{uniformly for}\quad t\in I,\\
\label{E:nm04}
&n_m'(t)=0\quad\text{or}\quad n_m(t)=n_1(t),\quad\text{for}\quad t\in I \\
\label{E:nm03}
&|n_m'(t)|\lesssim_u n_m(t)^3\quad\text{uniformly for}\quad t\in I, \\
\label{E:nm05}
&\int_I n_m^{-3}(t)\,dt\lesssim_u K, \\ 
\label{E:nm02}
&\int_I \frac{|n_m'(t)|}{n_m(t)^6}\,dt\lesssim_u 2^{-5m}K+1.
\end{align}
\end{lemma}

\begin{proof} Properties \eqref{E:nm01} and \eqref{E:nm04} follow by construction and \eqref{n101}. Similarly, as the algorithm decreases $|n_m'(t)|$ and increases $n_m(t)$ (as $m$ increases), property \eqref{E:nm03} follows from \eqref{n102}. Property \eqref{E:nm05} follows from \eqref{E:nm01} and \eqref{E:n02}. It remains to verify \eqref{E:nm02}.

As $n_m'=0$ on peaks and valleys, to compute the integral in \eqref{E:nm02} it suffices to sum over slopes, on which $n_m$ is monotone. By the fundamental theorem of calculus, on any slope $J$ we have
$$
\int_J \frac{|n_m'(t)|}{n_m(t)^6}\,dt=\tfrac15[v^{-5}-p^{-5}],
$$
where $v$ is the value of $n_m$ on the valley and $p$ is the value of $n_m$ on the peak. As the construction of $n_m$ from $n_{m-1}$ (i) decreases the total number of valleys (in the non-strict sense)  and (ii) doubles the value on each valley, it follows that
$$
\bigcup_{\text{slopes }J}\int_{J} \frac{|n_m'(t)|}{n_m(t)^6}\,dt\leq 2^{-5}\bigcup_{\text{slopes }J}\int_{J} \frac{|n_{m-1}'(t)|}{n_{m-1}(t)^6}\,dt.
$$
Thus \eqref{E:nm02} follows from induction and \eqref{n103}. (The additional $1$ on RHS\eqref{E:nm02} accounts for either end of $I$.) \end{proof}


\subsection{Choice of parameters}\label{S:parameters} We discuss here how to choose the parameters in the definition of the weight $a$. First of all, take $I$ large enough inside $[0,T_{max})$ to satisfy the hypotheses of the lemmas of Section~\ref{S:nt}. We now fix a small parameter $0<\alpha\ll 1$. In fact, $\alpha=\frac{1}{100}$ would suffice. The implicit constants below may depend on $\alpha$. 

Recalling the definition of $K$ in \eqref{E:QSK}, we choose $R$, $m$, and $J$ satisfiying 
\begin{align}
& R\sim K^\alpha, \quad e^J = R^\alpha, \quad 2^m \sim_u R^{\frac45(1+\alpha)}. \label{par1} 
\end{align}
We now define $a$ as in \eqref{D:a}. We choose $w$ as in \eqref{D:W} and \eqref{E:higher}, and we take $n(t)=n_m(t)$, with $m$ as in \eqref{par1} and $n_m$ as constructed in Section~\ref{S:nt}. 

Recall that by \eqref{E:KI}, we may make $K$ arbitrarily large by taking $I$ sufficiently large inside $[0,T_{max})$. Note also that
\begin{equation}\label{JlogR}
J\sim\log R.
\end{equation}

The choice of $m$ in \eqref{par1} is motivated by the estimation of \eqref{E:dta} in Lemma~\ref{L:dta} below; cf. \eqref{whym}. By \eqref{E:nm01}, $n_m$ may increase with $m$, while, on the other hand, $a=|x|$ only for $|x|\leq \frac{R}{n_m(t)}$. We need $a=|x|$ on a sufficiently large ball in order to get suitable lower bounds for $\int_I \frac{d}{dt}M(t;a)\,dt$. Using \eqref{par1}, \eqref{E:nm01}, and \eqref{E:n01}, we get the following lower bound:
\begin{equation}\label{radius}
\tfrac{R}{n_m(t)}\gtrsim_u \tfrac{R^{\frac15-\frac{4}{5}\alpha}}{N(t)}\gtrsim_u 
\tfrac{R^\alpha}{N(t)}\geq \tilde{c}(u) \tfrac{K^{\alpha^2}}{N(t)},
\end{equation}
provided $\alpha$ is sufficiently small. In particular, if $K$ is sufficiently large, we can guarantee $a=|x|$ on a large enough ball. By `large enough', we mean the following: for $C(u)$ sufficiently large (and $K$ sufficiently large depending on $u$), we can deduce the following lower bound: 
\begin{equation}\label{a lb}
\int_{|x-x(t)|\leq \frac{C(u)}{N(t)}}|\uhi(t,x)|^2\,dx\gtrsim_u N(t)^{-3}\quad\text{uniformly for}\quad t\in[0,T_{max}).
\end{equation}
Indeed, this is a consequence of almost periodicity and \eqref{E:Ng1} (and a few applications of H\"older, Bernstein, and Sobolev embedding). For details, refer to \cite[(7.3)]{MMZ}. With \eqref{radius} in mind, we now take $I$ possibly even larger inside $[0,T_{max})$ to guarantee that  
\begin{equation}\label{Klb}
\tilde{c}(u)K^{\alpha^2}>2C(u).
\end{equation}

The fact that we can choose $m$ both to control \eqref{whym} satisfactorily and to satisfy \eqref{radius} stems from the definition of $n_0$ in \eqref{D:n} and the property \eqref{E:nm04} of $n_m$; cf. \eqref{whym} below. 
 
Next, fix $0<\eta\ll 1$ to be chosen sufficiently small depending on $u$ below. We choose $I$ possibly even larger inside $[0,T_{max})$ so that
\begin{align}
 (\log K)^{-1}&\leq\eta.	\label{Kbig}
\end{align} 
In particular, by construction,
\begin{equation}\label{JJ0}
\tfrac{1}{J}\lesssim \eta. 
\end{equation}

We next use \eqref{APsmall} to choose $c=c(\eta)>0$ sufficiently small that
\begin{equation}\label{IMap}
\| \nsc P_{\leq cN(t)}u\|_{L_t^\infty L_x^2}\leq \eta.
\end{equation}

In what follows, we will use the following notation: 
\begin{equation}\label{E:lambda}
\lambda_j = \lambda_j(t)=\tfrac{Re^j}{n_m(t)}\quad\text{for}\quad j=0,\dots, J.
\end{equation}
Thus by construction, $a(t,x)=|x|$ for $|x|\leq\lambda_0(t)$ and $a(t,x)$ is constant for $|x|>\lambda_J(t)$. 

By \eqref{E:nm01}, \eqref{E:n01}, \eqref{par1}, and \eqref{JlogR}, we may choose $\alpha$ sufficiently small and take $I$ possibly even larger inside $[0,T_{max})$ to guarantee
\begin{equation}\label{E:Rbig}
\lambda_j \cdot cN(t)\geq \tfrac{J}{\eta}\quad\text{uniformly for}\quad j=0,\dots J\quad\text{and}\quad t\in I. 
\end{equation}

In the estimates below, we will encounter quantities of the form
$$
(Re^J)^\ell K^{\delta}\quad\text{for some}\quad \ell\geq 1\quad\text{and}\quad \delta\in(0,1).
$$
By choosing $\alpha$ sufficiently small and using \eqref{Kbig} and \eqref{par1}, we can guarantee that
\begin{equation}\label{Vsmall}
(Re^J)^\ell K^{\delta}\lesssim K^{1-\alpha}\lesssim \eta K
\end{equation}
for all combinations of $\ell,\delta$ appearing below. In particular, by \eqref{E:nm01} and \eqref{E:boundM}, we have
\begin{equation}\label{E:boundM2}
\sup_{t\in I}|M(t;a)|\lesssim_u \eta K.
\end{equation}

\subsection{Estimation of \eqref{E:dta} through \eqref{E:massmass}} Having chosen parameters, we turn to estimating the terms appearing in Proposition~\ref{P:identity}.

\begin{lemma}[Estimation of \eqref{E:dta}]\label{L:dta}
\begin{equation}\label{est:dta}
\biggl\vert\iiint |u(y)|^2 \partial_ta_k(x-y) 2\Im \bar{u}u_k(x)\,dx\,dy\,dt\biggr|\lesssim_u \eta K.
\end{equation}
\end{lemma}
\begin{proof} Note that
$$
\partial_t a_k(t,x)=w_{rr}(n_m(t)|x|) x_k n_m'(t).
$$
Thus, by \eqref{E:higher},  
$$
\|\partial_t \nabla a\|_{L_x^1}\lesssim \tfrac{1}{J}(Re^J)^4 \tfrac{|n_m'(t)|}{n_m(t)^5}.
$$

For the low frequency contributions, Proposition~\ref{P:decay}, Sobolev embedding, Corollary~\ref{C:lts}, \eqref{Vsmall}, \eqref{E:nm03}, and \eqref{E:nm05} imply
\begin{align*}
\biggl|&\iiint |\ulo(y)|^2\partial_t a_k(x-y)2\Im \bar{u}u_k(x)\,dx\,dy\,dt\biggr| \\
&\lesssim\bigl[\|u\|_{L_t^\infty L_x^{\frac{72}{19}}}^3 \|\nabla\ulo\|_{L_t^3 L_x^{\frac{24}{5}}} + \|\ulo\|_{L_t^4 L_x^\infty}\|u\|_{L_t^\infty L_x^{\frac{48}{13}}}^2\|\nabla\uhi\|_{L_t^{12} L_x^{\frac{24}{11}}}\bigr]\|\partial_t\nabla a\|_{L_t^{\frac32} L_x^1}  \\
&\lesssim_u \bigl[\|\nsc \ulo\|_{L_{t,x}^3}+K^{\frac{1}{10}}\bigr](Re^J)^4\biggl(\int \biggl(\frac{|n_m'(t)|}{n_m(t)^5}\biggr)^{\frac32}\,dt\biggr)^{\frac23} \\
&\lesssim_u K^{\frac{1}{10}}(Re^J)^4\biggl(\int_I n_m(t)^{-3}\,dt\biggr)^{\frac23}\lesssim_u (Re^J)^4K^{\frac{23}{30}}\lesssim_u\eta K.
\end{align*}

For the high frequency terms, we use the definition of $n_0$ in \eqref{D:n}. We also rely crucially on \eqref{E:nm02} and \eqref{n101}. Using \eqref{par1}, \eqref{JJ0}, and \eqref{Vsmall} as well, 
\begin{align}
\nonumber
\biggl|\iiint |\uhi(y)|^2&\partial_t a_k(x-y)2\Im \bar{u}u_k(x)\,dx\,dy\,dt\biggr| \\
\nonumber
&\lesssim \int_I \|\uhi(t)\|_{L_x^4}^2 \|\partial_t \nabla a\|_{L_x^1} \|u(t)\|_{L_x^8} \|\nabla u(t)\|_{L_x^{\frac83}}\,dt \\
\nonumber
&\lesssim_u \tfrac{1}{J}(Re^J)^4 \int_I \frac{|n_m'(t)|}{n_0(t) n_m(t)^5}\,dt \\
&\lesssim_u \tfrac{1}{J}(Re^J)^4 (2^{-5m}K+1)\lesssim_u \eta K.	\label{whym}
\end{align}
This completes the proof of Lemma~\ref{L:dta}.\end{proof}

Before turning to the other error terms, we combine Proposition~\ref{L:mass} with Corollary~\ref{C:lts} to deduce the following estimate, which inherits the logarithmic loss from Proposition~\ref{L:mass}.

\begin{corollary}\label{C:mass} The following estimate holds: 
$$
\sup_{j\leq J}\int_I \sup_{x\in\R^4} \int_{|x-y|\leq \lambda_j} |\uhi(t,y)|^2\,dy\,dt\lesssim_u JK,
$$
where $\lambda_j$ is as in \eqref{E:lambda}.
\end{corollary}

\begin{proof} We apply Proposition~\ref{L:mass} with $n=n_m$, $\lambda=\lambda_j$, and $u=\uhi$, choosing $F=\PHi\text{\O}(\ulo^2 u^3)$ and $G=\text{\O}(\uhi^2 u^3)$. 

First, by \eqref{E:nm01}, \eqref{Vsmall}, Bernstein,  and Corollary~\ref{C:lts}, we have
\begin{align*}
\|\lambda_j\|_{L_t^\infty}^2&\bigl[\|\uhi\|_{L_t^\infty L_x^2}^2+\|F\|_{L_t^2 L_x^{\frac43}}^2\bigr]
\lesssim_u (Re^J)^2 K^{\frac35}\lesssim_u K.
\end{align*}

Next, using \eqref{JlogR} and Corollary~\ref{C:lts},
$$
\sup_{0\leq j\leq J}[1+\|\log(\lambda_j n_m)\|_{L_t^\infty}]\|G\|_{L_t^2 L_x^1}^2 \lesssim_u JK.
$$
Finally, by \eqref{E:bdd} and \eqref{E:nm05}, we have
$$
\|G\|_{L_t^\infty L_x^{\frac85}}^2\int_I n_m(t)^{-3}\,dt\lesssim_u K.
$$
The result follows.\end{proof}

We now turn to \eqref{E:scary}.

\begin{lemma}[Estimation of \eqref{E:scary}]\label{L:scary} Define
$$
\Phi_{jk}(x,y)=|u(y)|^2 \Re(\bar{u}_k u_j)(x)-\Im(\bar{u}u_j)(y)\Im(\bar{u}u_k)(x).
$$
We estimate \eqref{E:scary} in two pieces:
\begin{align}
\label{scary1}
& \iiint\limits_{|x-y|\leq\lambda_0} a_{jk}(x-y)\Phi_{jk}(x,y)\,dx\,dy\,dt\geq 0, \\
\label{scary2}
&\biggl|\iiint\limits_{|x-y|>\lambda_0}a_{jk}(x-y)\Phi_{jk}(x,y)\,dx\,dy\,dt\biggr| \lesssim_u \eta K.
\end{align}
\end{lemma}

\begin{proof} As $a_{jk}(x-y)=a_{jk}(y-x)$, we may replace $\Phi_{jk}$ by the hermitian matrix
$$
\Psi_{jk}(x,y)=\tfrac12\Phi_{jk}(x,y)+\tfrac12\Phi_{jk}(y,x).
$$
We have that $\Psi_{jk}$ is a positive semi-definite quadratic form on $\R^4$, since
$$
|e_ke_j\Im(\bar{u}u_j)(y)\Im(\bar{u}u_k)(x)|\leq\tfrac12|u(x)|^2|e\cdot\nabla u(y)|^2+\tfrac12|u(y)|^2|e\cdot\nabla u(x)|^2
$$
for any $e\in\R^4$. Recalling that $a=|x|$ for $|x|\leq\lambda_0$, we see that $a_{jk}$ is positive semi-definite for $|x|\leq\lambda_0$. Thus \eqref{scary1} follows. 

In general, the eigenvalues of the Hessian of $a$ are $a_{rr}$ and $\tfrac{a_r}{r}$. By construction, we have $a_r\geq 0$ and $|a_{rr}|\lesssim\tfrac{1}{Jr}$. Thus, to prove \eqref{scary2}, it suffices to estimate
\begin{equation}\label{scary3}
\int_I \iint_{\lambda_0 < |x-y| \leq \lambda_J} \frac{|\nabla u(x)|^2 |u(y)|^2}{J|x-y|}\,dx\,dy\,dt.
\end{equation}

Taking $c=c(\eta)$ as in \eqref{IMap}, we have
\begin{align}
\eqref{scary3} &\lesssim \iiint\limits_{\lambda_0<|x-y|\leq\lambda_J} \frac{|\nabla u(x)|^2|u(y)|^2-|\nabla\uhi(x)|^2|\uhi(y)|^2}
	{|x-y|}\,dx\,dy\,dt \label{scary31} \\
\label{scary32}
&+\tfrac{1}{J}\sum_{j=0}^J \iiint\limits_{\lambda_j<|x-y|\leq\lambda_{j+1}}\lambda_j^{-1}|\Pj\nabla\uhi(x)|^2
	|\uhi(y)|^2\,dx\,dy\,dt \\
\label{scary32.5}
&+\tfrac{1}{J}\sum_{j=0}^J \iiint\limits_{\lambda_j<|x-y|\leq\lambda_{j+1}}\lambda_j^{-1}|\Pjm\nabla\uhi(x)|^2
	|\uhi(y)|^2\,dx\,dy\,dt \\
\label{scary33}
&+ \tfrac{1}{J}\sum_{j=0}^J \iiint\limits_{\lambda_j<|x-y|\leq\lambda_{j+1}}\lambda_j^{-1}
	|\Pc\nabla\uhi(x)|^2|\uhi(y)|^2\,dx\,dy\,dt.
\end{align}

We first use Corollary~\ref{C:lts}, Proposition~\ref{P:decay}, \eqref{E:nm01}, \eqref{E:nm05}, and \eqref{Vsmall} to estimate
\begin{align*}
\eqref{scary31}&\lesssim \bigl[\|\nabla\ulo\|_{L_{t,x}^4} \|\nabla u\|_{L_t^\infty L_x^{\frac83}}\|u\|_{L_t^\infty L_x^8} \|u\|_{L_t^\infty L_x^4} \\
&\quad\quad +\|\nabla u\|_{L_t^\infty L_x^{\frac83}}^2 \|\ulo\|_{L_t^4 L_x^\infty} \|u\|_{L_t^\infty L_x^4}\bigr]\biggl(\int_I n_m(t)^{-4}\,dt\biggr)^{\frac34}(Re^J)^3 \\
&\lesssim_u \bigl[\|\nsc \ulo\|_{L_t^4 L_x^{\frac83}}+1\bigr]K^{\frac34}(Re^J)^3 \lesssim_u \eta K. 
\end{align*}

For \eqref{scary32}, we let
$$
S=\{N_1,\dots,N_4, j\,|\, K^{-\frac15}\leq N_1\leq\cdots\leq N_4,\ N_2\leq \lambda_j^{-1}\}
$$
and use Young's inequality, H\"older's inequality, Bernstein, Cauchy--Schwarz, and Corollary~\ref{C:lts} to estimate
\begin{align*}
\eqref{scary32}&\lesssim\tfrac{1}{J}\int_I \bigl[\sup_M M^{-2}\|P_M\uhi\|_{L_x^\infty}\bigr]^2\sum_S \lambda_j^{3}
	\tfrac{(N_1N_2)^2}{(N_3 N_4)^{\frac12}}\|u_{N_3}\|_{\dot H_x^{\frac32}} \|u_{N_4}\|_{\dot H_x^{\frac32}}\,dt \\
&\lesssim_u \tfrac{1}{J} \|\nsc u\|_{L_t^\infty L_x^2}^2 K \lesssim_u \eta K.
\end{align*}

For \eqref{scary32.5}, we use Bernstein, \eqref{IMap}, and Corollary~\ref{C:mass} to estimate
\begin{align}
&\eqref{scary32.5}\nonumber \\
&\lesssim \tfrac1J\biggl\|\sum_{j=0}^J 
	\|\lambda_j^{-\frac12}\Pjm\nabla\uhi\|_{ L_x^2}^2\biggr\|_{L_t^\infty}
\sup_{j\leq J}\int_I \sup_{x\in\R^4}\int_{|x-y|\leq\lambda_{j+1}}|\uhi(y)|^2\,dy\,dt \nonumber \\
&\lesssim \|\nsc P_{\leq cN(t)}u\|_{L_t^\infty L_x^2}^2 K\lesssim_u \eta K.	\label{est scary32.5}
\end{align}

We next turn to \eqref{scary33}. In light of Corollary~\ref{C:mass}, it suffices to show that 
\begin{equation}\label{scary331}
\sum_{j=0}^J \|\lambda_j^{-\frac12}\Pc\nabla\uhi\|_{L_t^\infty L_x^2}^2\lesssim_u \eta.
\end{equation}

In fact, by Bernstein and \eqref{E:Rbig}, we have
\begin{align*}
\|\lambda_j^{-\frac12}\Pc\nabla\uhi\|_{L_x^2}^2\lesssim [\lambda_j cN(t)]^{-1}\|\nsc u\|_{L_t^\infty L_x^2}^2\lesssim_u \tfrac{\eta}{J}
\end{align*} 
uniformly in $j$ and $t$. Thus \eqref{scary331} follows. This completes the proof of Lemma~\ref{L:scary}. \end{proof}

We next turn to \eqref{E:potential}. 

\begin{lemma}[Potential energy term]\label{L:potential} We estimate \eqref{E:potential} as follows. First,
\begin{align}
\label{E:potentiallow}
&\biggl|\iiint \Delta a(x-y)\bigl[|u(x)|^6|u(y)|^2-|\uhi(x)|^6|\uhi(y)|^2\bigr]\,dx\,dy\,dt\biggr|\lesssim_u \eta K, \\
\label{E:pos1}
&\iiint_{|x-y|\leq\lambda_0} \Delta a(x-y)|\uhi(x)|^6 |\uhi(y)|^2\,dx\,dy\,dt\geq 0.
\end{align}
Next, recalling $\Delta a=a_{rr}+\tfrac{3}{r}a_r$, we have
\begin{align}
\label{E:pos2}
&\iiint_{|x-y|>\lambda_0} \frac{a_r(x-y) }{|x-y|}|\uhi(x)|^6 |\uhi(y)|^2\,dx\,dy\,dt\geq 0, \\
\label{E:potsmall}
&\biggl\vert\iiint_{|x-y|>\lambda_0} a_{rr}(x-y)|\uhi(x)|^6 |\uhi(y)|^2 \,dx\,dy\,dt\biggr\vert\lesssim_u \eta K.
\end{align}
\end{lemma} 

\begin{proof} First, by Corollary~\ref{C:lts}, Proposition~\ref{P:decay}, \eqref{E:nm01}, \eqref{E:nm05}, and \eqref{Vsmall}, we have
\begin{align*}
\eqref{E:potentiallow}\lesssim 
\|\ulo\|_{L_t^4 L_x^\infty}\|u\|_{L_t^\infty L_x^7}^7 \biggl(\int_I n_m(t)^{-4}\,dt\biggr)^{\frac34}(Re^J)^3
\lesssim_u K^{\frac34}(Re^J)^3\lesssim_u \eta K.
\end{align*}

Next, we have by construction that $a=|x|$ for $|x|\leq\lambda_0$ and $a_r\geq 0$ for all $x$. Thus, \eqref{E:pos1} and \eqref{E:pos2} follow.

Finally, we turn to \eqref{E:potsmall}. By \eqref{E:higher}, it suffices to estimate 
\begin{equation}\label{E:potsmall2} 
\int_I\iint_{\lambda_0<|x-y|\leq\lambda_J}\frac{|\uhi(x)|^6 |\uhi(y)|^2}{J|x-y|}\,dx\,dy\,dt.
\end{equation} 

To this end, we proceed as in Lemma~\ref{L:scary} and write
\begin{align}
\eqref{E:potsmall2}&\lesssim \tfrac{1}{J}\sum_{j=0}^J \iiint\limits_{\lambda_j<|x-y|\leq\lambda_{j+1}}	
	\lambda_j^{-1} |\Pj\uhi(x)|^6 |\uhi(y)|^2 \,dx\,dy\,dt \label{E:pot1} \\
&+ \tfrac{1}{J}\sum_{j=0}^J \iiint\limits_{\lambda_j<|x-y|\leq\lambda_{j+1}}
	\lambda_j^{-1}|\Pjm\uhi(x)|^6 |\uhi(y)|^2 \,dx\,dy\,dt \label{E:pot2} \\
&+ \tfrac{1}{J}\sum_{j=0}^J \iiint\limits_{\lambda_j<|x-y|\leq\lambda_{j+1}}
	\lambda_j^{-1}|\Pc\uhi(x)|^6|\uhi(y)|^2\,dx\,dy\,dt,	\label{E:pot3} 
\end{align}
where $c$ is as in \eqref{IMap}. 

For \eqref{E:pot1}, we let
$$
S=\{N_1,\dots,N_6,j\ |\, K^{-\frac15}\leq N_1\leq\cdots\leq N_6\leq(\lambda_j)^{-1},\ j=0,\dots,J\}.
$$
Then by Young's inequality, H\"older's inequality, Bernstein, Cauchy--Schwarz, Corollary~\ref{C:lts}, and \eqref{JJ0},  
\begin{align*}
\eqref{E:pot1}&\lesssim \tfrac1J\|u\|_{L_t^\infty L_x^8}^4\int_I\sum_S \lambda_j^{3}\|u_{N_1}\|_{L_x^\infty}
	\|u_{N_2}\|_{L_x^\infty} \|u_{N_5}\|_{L_x^4} \|u_{N_6}\|_{L_x^4}\,dt \\
&\lesssim_u \tfrac1J\int_I \bigl[\sup_M M^{-2}\|P_M\uhi\|_{L_x^\infty}\bigr]^2\sum_S \lambda_j^{3}
	\tfrac{(N_1N_2)^2}{(N_5 N_6)^{\frac12}}\|u_{N_5}\|_{\dot H_x^{\frac32}} \|u_{N_6}\|_{\dot H_x^{\frac32}}\,dt \\
&\lesssim_u \tfrac1J\|\nsc u\|_{L_t^\infty L_x^2}^2K\lesssim_u \eta K.
\end{align*}

For \eqref{E:pot2} and \eqref{E:pot3}, we first note the following consequence of Sobolev embedding:
$$
\|\uhi(t)\|_{L_x^6}^6\lesssim \|\uhi(t)\|_{L_x^4}^2\|\uhi(t)\|_{L_x^8}^4\lesssim_u \|\nabla \uhi(t)\|_{L_x^2}^2.
$$
Thus we can estimate \eqref{E:pot2} as we did \eqref{scary32.5}; cf \eqref{est scary32.5}. Similarly, we can estimate \eqref{E:pot3} as we did \eqref{scary33}. This completes the proof of Lemma~\ref{L:potential}. \end{proof}

Finally, we turn to \eqref{E:massmass}. 
 
\begin{lemma}[Mass-mass term]\label{L:massmass} We estimate the contribution of \eqref{E:massmass} in three pieces:
\begin{align}
\label{E:A}
&\iiint\limits_{|x-y|\leq\lambda_0}-\Delta\Delta a(x-y)|\uhi(x)|^2|\uhi(y)|^2\,dx\,dy\,dt\geq 0, \\
\label{mm1}
&\biggl\vert\iiint \Delta\Delta a(x-y)\bigl[|u(x)|^2|u(y)|^2-|\uhi(x)|^2|\uhi(y)|^2\bigr]\,dx\,dy\,dt\biggr\vert \lesssim_u \eta K, \\
\label{mm3}
&\biggl|\iiint\limits_{|x-y|>\lambda_0}\Delta\Delta a(x-y)|\uhi(x)|^2|\uhi(y)|^2\,dx\,dy\,dt\biggr| \lesssim_u 
	\eta K.
\end{align}
\end{lemma}

\begin{remark} The term appearing in \eqref{E:A} will give the left-hand side of the interaction Morawetz inequality in Proposition~\ref{P:IM}.\end{remark}

\begin{proof}  As $a=|x|$ for $|x|\leq\lambda_0$, we have $-\Delta\Delta a=3|x|^{-3}$ in this region. Thus \eqref{E:A} holds.

For \eqref{mm1}, we use the fact that $\partial_{y_k}a(t,x-y)=-\partial_{x_k}a(t,x-y)$ and integrate by parts. Specifically, writing
$$
L(x,y)=|u(x)|^2|u(y)|^2-|\uhi(x)|^2|\uhi(y)|^2
$$
and integrating by parts, we are left to estimate
$$
\iiint\limits_{|x-y|\leq\lambda_J} \frac{\bigl|\partial_{y_k}\partial_{x_k} L(x,y)\bigr|}{|x-y|}\,dx\,dy\,dt.
$$
However, expanding $L(x,y)$ and applying the derivatives, this term can be seen to be a sum of the types of terms already estimated when dealing with \eqref{scary31}. Thus \eqref{mm1} holds.

For \eqref{mm3}, we again proceed as in Lemma~\ref{L:scary}. In particular, 
\begin{align}
\eqref{mm3}&\lesssim \tfrac{1}{J}\sum_{j=0}^J \iiint\limits_{\lambda_j<|x-y|\leq\lambda_{j+1}} \lambda_j^{-3}|\Pj\uhi(x)|^2 |\uhi(y)|^2\,dx\,dy\,dt \label{mm31} \\
\label{mm32}
&+\tfrac{1}{J}\sum_{j=0}^J \iiint\limits_{\lambda_j<|x-y|\leq\lambda_{j+1}} \lambda_j^{-3}|\Pjm\uhi(x)|^2 |\uhi(y)|^2\,dx\,dy\,dt \\
\label{mm33}
&+\tfrac{1}{J}\sum_{j=0}^J \iiint\limits_{\lambda_j<|x-y|\leq\lambda_{j+1}} \lambda_j^{-3}|\Pc\uhi(x)|^2 |\uhi(y)|^2\,dx\,dy\,dt,
\end{align}
where $c$ is as in \eqref{IMap}. 

For \eqref{mm31}, define
$$
S=\{N_1,\dots,N_4,j\ |\, K^{-\frac15}\leq N_1\leq N_2\leq N_3 \leq N_4,\ N_2\leq(\lambda_j)^{-1},\ j=0,\dots,J\}.
$$
By Young's inequality, H\"older's inequality, Bernstein, Cauchy--Schwarz, Corollary~\ref{C:lts}, and \eqref{JJ0}, 
\begin{align*}
\eqref{mm31}&\lesssim \tfrac{1}{J}\int_I \sum_S \lambda_j \|u_{N_1}\|_{L_x^\infty} \|u_{N_2}\|_{L_x^\infty}
	\|u_{N_3}\|_{L_x^2}\|u_{N_4}\|_{L_x^2}\,dt \\
&\lesssim\tfrac{1}{J}\int_I \bigl[\sup_M M^{-2}\|P_M\uhi\|_{L_x^\infty}\bigr]^2\sum_S \lambda_j
	\tfrac{(N_1N_2)^2}{(N_3 N_4)^{\frac32}} \|u_{N_3}\|_{\dot H_x^{\frac32}}\|u_{N_4}\|_{\dot H_x^{\frac32}}\,dt \\
&\lesssim_u \tfrac{1}{J} \|u\|_{L_t^\infty \dot H_x^{\frac32}}^2 K \lesssim_u \eta K.
\end{align*}

We can now estimate \eqref{mm32} and \eqref{mm33} by proceeding just as we did \eqref{scary32.5} and \eqref{scary33}. In particular, the estimation of \eqref{mm32} relies on Bernstein and \eqref{IMap}, while the estimation of \eqref{mm33} relies on Bernstein and \eqref{E:Rbig}. This completes the proof of \eqref{L:massmass}.  \end{proof}


\subsection{Interaction Morawetz inequality}\label{S:IM}

We now collect our estimates to deduce the following interaction Morawetz inequality. 

\begin{proposition}[Interaction Morawetz]\label{P:IM} Suppose $u:[0,T_{max})\times\R^4\to\C$ is an almost periodic solution to \eqref{nls} as in Theorem~\ref{T:AP} such that \eqref{E:KI} holds. 

Let $I\subset[0,T_{max})$ be a compact time interval, which is a contiguous union of characteristic subintervals. Define $K$ as in \eqref{E:QSK} and $\uhi$ as in \eqref{def:hi}.

Let $0<\eta\ll 1$. For $I$ sufficiently large inside $[0,T_{max})$,
\begin{equation}\label{E:IM}
\int_I \iint\limits_{|x-y|\leq\frac{2C(u)}{N(t)}}\frac{|\uhi(x)|^2|\uhi(y)|^2}{|x-y|^3}\,dx\,dy\,dt\lesssim_u \eta K,
\end{equation}
where $C(u)$ is as in \eqref{a lb}. The implicit constant does not depend on $I$. 
\end{proposition}

\begin{proof} We set up the interaction Morawetz action $M(t;a)$ as in Section~\ref{S:setup}, constructing the rescaling function $n(t)$ as in Section~\ref{S:nt} and choosing parameters as in Section~\ref{S:parameters}. By the fundamental theorem of calculus and \eqref{E:boundM2}, we have
$$
\int_I \tfrac{d}{dt}M(t;a)\,dt\lesssim_u \eta K.
$$
We now recall the identity in Proposition~\ref{P:identity} and collect the estimates in Lemma~\ref{L:dta} and Lemmas~\ref{L:scary}--\ref{L:massmass}. Holding on to the term appearing in \eqref{E:A}, we deduce
\begin{equation}\label{collect1}
\int_I \iint\limits_{|x-y|\leq\lambda_0} \frac{|\uhi(x)|^2 |\uhi(y)|^2}{|x-y|^3}\,dx\,dy\,dt\lesssim_u \eta K.
\end{equation}
Using the definition of $\lambda_0$ in \eqref{E:lambda}, as well as \eqref{radius} and \eqref{Klb}, we deduce \eqref{E:IM}. \end{proof}

We can now complete the proof of Theorem~\ref{T:QS}.

\begin{proof}[Proof of Theorem~\ref{T:QS}] Suppose $u:[0,T_{max})\times\R^4\to\C$ is an almost periodic solution as in Theorem~\ref{T:AP} such that \eqref{E:KI} holds. Let $I$, $K$, $\eta$, and $C(u)$ be as in Proposition~\ref{P:IM}. By \eqref{a lb},
\begin{align}
\int_I \iint\limits_{|x-y|\leq\frac{2C(u)}{N(t)}} \frac{|\uhi(x)|^2|\uhi(y)|^2}{|x-y|^3}\,dx\,dy\,dt&
	 \gtrsim_u \int_I N(t)^3\biggl(\int_{|x-x(t)|\leq\frac{C(u)}{N(t)}} |\uhi(x)|^2\,dx\biggr)^2\,dt \nonumber \\
&\gtrsim_u \int_I N(t)^{-3}\,dt\gtrsim_u K. \label{LB}
\end{align}
Combining \eqref{E:IM} and \eqref{LB} yields
$$
K\lesssim_u \eta K.
$$
We now choose $\eta=\eta(u)$ sufficiently small to deduce $K=0$, which is a contradiction. This completes the proof of Theorem~\ref{T:QS}. \end{proof}

\begin{center}

\end{center}


\begin{thebibliography}{99}


\bibitem{BeVa} P. B\'egout and A. Vargas, \emph{Mass concentration phenomena for the $L^2$-critical nonlinear Schr\"odinger equation}. Trans. Amer. Math. Soc. \textbf{359} (2007), no. 11, 5257--5282. MR2327030

\bibitem{Bo99a} J. Bourgain, \emph{Global wellposedness of defocusing critical nonlinear Schr\"odinger equation in the radial case}. J. Amer. Math. Soc. \textbf{12} (1999), no. 1, 145--171. MR1626257

\bibitem{CarKer} R. Carles and S. Keraani, \emph{On the role of quadratic oscillations in nonlinear Schr\"odinger equations. II.
 The $L^2$-critical case.} Trans. Amer. Math. Soc. \textbf{359} (2007), no. 1, 33--62. MR2247881

\bibitem{Cav} T. Cazenave, \emph{Semilinear Schr\"odinger equations.} Courant Lecture Notes in Mathematics \textbf{10}, 2003. MR2002047

\bibitem{CaW} T. Cazenave and F. Weissler, \emph{The Cauchy problem for the critical nonlinear Schr\"odinger equation in $H^s$.} Nonlinear Anal. \textbf{14} (1990), no. 10, 807--836. MR1055532

\bibitem{CW} M. Christ and M. Weinstein, \emph{Dispersion of small amplitude solutions of the generalized Korteweg-de Vries equation.} J. Funct. Anal. \textbf{100} (1991), no. 1, 87--109. MR1124294

\bibitem{ckstt} J. Colliander, M. Keel, G. Staffilani, H. Takaoka, and T. Tao, \emph{Global existence and scattering for rough solutions of a nonlinear Schr\"odinger equation on $\R^3$}. Comm. Pure Appl. Math. \textbf{57} (2004), no. 8, 987--1014. MR2053757

\bibitem{CKSTT07}J. Colliander, M. Keel, G. Staffilani, H. Takaoka, and T. Tao, \emph{Global well-posedness and scattering for the energy-critical nonlinear Schr\"{o}dinger equation in $\mathbb{R}^3$}. Ann. Math. (2) \textbf{167} (2008), no. 3, 767--865. MR2415387

\bibitem{Dodson3} B. Dodson, \emph{Global well-posedness and scattering for the defocusing, $L^2$-critical nonlinear Schr\"odinger equation when $d\geq3$}. J. Amer. Math. Soc. \textbf{25} (2012), no. 2, 429--463. MR2869023

\bibitem{Dodson2} B. Dodson, \emph{Global well-posedness and scattering for the defocusing, $L^2$-critical nonlinear Schr\"odinger equation when $d = 2$.} Preprint {\tt arXiv:1006.1375}

\bibitem{Dodson1} B. Dodson, \emph{Global well-posedness and scattering for the defocusing, $L^2$-critical nonlinear Schr\"odinger equation when $d = 1$.} Preprint {\tt arXiv:1010.0040}

\bibitem{Dodson} B. Dodson, \emph{Global well-posedness and scattering for the mass critical nonlinear Schr\"odinger equation with mass below the mass of the ground state.} Preprint {\tt arXiv:1104.1114}

\bibitem{Dodson14} B. Dodson, \emph{Global well-posedness and scattering for the focusing, energy-critical nonlinear Schr\"odinger problem in dimension $d=4$ for initial data below a ground state threshold}  Preprint {\tt arXiv:1409.1950v1}

\bibitem{GV} J. Ginibre and G. Velo, \emph{Smoothing properties and retarded estimates for some dispersive evolution equations.} Comm. Math. Phys. \textbf{144} (1992), no. 1, 163--188. MR1151250

\bibitem{Grillakis} G. Grillakis, \emph{On nonlinear Schr\"odinger equations}. Comm. Partial Differential Equations \textbf{25} (2000), no. 9--10, 1827--1844. MR1778782

\bibitem{HolRou} J. Holmer and S. Roudenko, \emph{A sharp condition for scattering of the radial 3D cubic nonlinear Schr\"odinger equation.} Comm. Math. Phys. \textbf{282} (2008), no. 2, 435--467. MR2421484

\bibitem{KeT98} M. Keel and T. Tao, \emph{Endpoint Strichartz estimates.} Amer. J. Math. \textbf{120} (1998), no. 5, 955--980. MR1646048

\bibitem{KM} C. Kenig and F. Merle, \emph{Global well-posedness, scattering, and blow-up for the energy-critical, focusing, nonlinear Schr\"{o}dinger equation in the radial case.} Invent. Math. \textbf{166} (2006), no. 3, 645---675. MR2257393

\bibitem{KM2010} C. Kenig and F. Merle, \emph{Scattering for $\dot{H}^{1/2}$ bounded solutions to the cubic, defocusing NLS in $3$ dimensions.} Trans. Amer. Math. Soc. \textbf{362} (2010), no. 4, 1937--1962. MR2574882

\bibitem{Ker0} S. Keraani, \emph{On the defect of compactness for the Strichartz estimates for the Schr\"odinger equations}. J. Differential Equations \textbf{175} (2001), no. 2, 353--392. MR1855973

\bibitem{Ker} S. Keraani, \emph{On the blow up phenomenon of the critical nonlinear Schr\"odinger equation.} J. Funct. Anal. \textbf{235} (2006), no. 1, 171--192. MR2216444

\bibitem{KTV2009} R. Killip, T. Tao, and M. Visan, \emph{The cubic nonlinear Schr\"odinger equation in two dimensions with radial data.} J. Eur. Math. Soc. (JEMS) \textbf{11} (2009), no. 6, 1203--1258. MR2557134

\bibitem{KV2010} R. Killip and M. Visan, \emph{Energy-supercritical NLS: critical $\dot{H}^s$-bounds imply scattering.} Comm. Partial Differential Equations \textbf{35} (2010), no. 6, 945--987. MR2753625.

\bibitem{KV20101} R. Killip and M. Visan, \emph{The focusing energy-critical nonlinear Schr\"odinger equation in dimensions five and higher.} Amer. J. Math. \textbf{132} (2010), no. 2, 361--424. MR2654778.

\bibitem{KVnote} R. Killip and M. Visan, \emph{Nonlinear Schr\"odinger equations at critical regularity.} Clay Math. Proc., \textbf{17} (2013), 325--437. MR3098643

\bibitem{KV3D} R. Killip and M. Visan, \emph{Global well-posedness and scattering for the defocusing quintic NLS in three dimensions.} Anal. PDE \textbf{5} (2012), no. 4, 855--885. MR3006644

\bibitem{KVZ2008} R. Killip, M. Visan, and X. Zhang, \emph{The mass-critical nonlinear Schr\"odinger equation with radial data in dimensions three and higher.} Anal. PDE \textbf{1} (2008), no. 2, 229--266. MR2472890.

\bibitem{MerVeg} F. Merle and L. Vega, \emph{Compactness at blow-up time for $L^2$ solutions of the critical nonlinear Schr\"odinger equation in 2D}. Internat. Math. Res. Notices 1998, no. 8, 399--425. MR1628235

\bibitem{MMZ} C. Miao, J. Murphy, and J. Zheng, \emph{The defocusing energy-supercritical {NLS} in four space dimensions}. J. Funct. Anal. \textbf{267} (2014), no. 6, 1662--1724. MR3237770

\bibitem{Mu} J. Murphy, \emph{Intercritical NLS: critical $\dot{H}^s$-bounds imply scattering.} SIAM J. Math. Anal. \textbf{46} (2014), no. 1, 939--997. MR3166962

\bibitem{Mu2} J. Murphy, \emph{The $\dot{H}_x^{1/2}$-critical NLS in high dimensions}. DCDS-A \textbf{34} (2014), no. 2, 733--748. MR3094603

\bibitem{Mu3} J. Murphy, \emph{The radial defocusing nonlinear Schr\"odinger equation in three space dimensions}. Comm. Partial Differential Equations \textbf{40} (2015), no. 2, 265--308. MR3277927

\bibitem{RV} E. Ryckman and M. Visan, \emph{Global well-posedness and scattering for the defocusing energy-critical nonlinear Schr\"odinger equation in $\R^{1+4}$}. Amer. J. Math. \textbf{129} (2007), no. 1, 1--60. MR2288737

\bibitem{Shao} S. Shao, \emph{Maximizers for the Strichartz and Sobolev-Strichartz inequalities for the Schr\"odinger equation}. Electron. J. Differential Equations 2009, No. 3, 13pp. MR2471112

\bibitem{St} R. S. Strichartz, \emph{Restriction of Fourier transform to quadratic surfaces and decay of solutions of wave equations.} Duke Math. J. \textbf{44} (1977), no. 3, 705--714. MR0512086.

\bibitem{TaoRadial} T. Tao, \emph{Global well-posedness and scattering for higher-dimensional energy-critical nonlinear Schr\"odinger equation for radial data.} New York J. Math. \textbf{11} (2005), 57--80. MR2154347

\bibitem{TVZ2007} T. Tao, M. Visan, and X. Zhang, \emph{Global well-posedness and scattering for the defocusing mass-critical nonlinear Schr\"odinger equation for radial data in high dimensions.} Duke Math. J. \textbf{140} (2007), no. 1, 165--202. MR2355070.

\bibitem{Visanphd} M. Visan, \emph{The defocusing energy-critical nonlinear Schr\"odinger equation in dimensions five and higher.} Ph.D Thesis, UCLA, 2006. MR2709575

\bibitem{Visan2007} M. Visan, \emph{The defocusing energy-critical nonlinear Schr\"odinger equation in higher dimensions.} Duke Math. J. \textbf{138} (2007), no. 2, 281--374. MR2318286

\bibitem{Visan2011} M. Visan, \emph{Global well-posedness and scattering for the defocusing cubic nonlinear Schr\"odinger equation in four dimensions.} Int. Math. Res. Not. IMRN 2012, no. 5, 1037--1067. MR2899959

\bibitem{Visbook} M. Visan, \emph{Dispersive Equations}. In ``Dispersive Equations and Nonlinear Waves'', Oberwolfach Seminars \textbf{45}, Birkh\"auser/Springer Basel 2014. 

\bibitem{XF} J. Xie and D. Fang, \emph{Global well-posedness and scattering for the defocusing $\dot H^s$-critical NLS.} Chin. Ann. Math. Ser. B \textbf{34} (2013), no. 6, 801--842. MR3122297

\end{thebibliography}
\end{document}